\documentclass[11pt]{amsart}
\usepackage{latexsym}
\usepackage{graphicx, color, epsfig, verbatim, subfigure, epsf} 
\usepackage{amssymb, amsmath,mathrsfs}
\usepackage[left=2.5cm,right=2.5cm,top=2.5cm, bottom=2.5cm ]{geometry}
\usepackage{graphicx}
\usepackage{tikz}
\usepackage{pgfplots}

\usepackage{float}

\usepackage{setspace}
\usepackage[pagebackref,hypertexnames=false, colorlinks, citecolor=red, linkcolor=red]{hyperref}

\newcommand{\McC}{\raise.5ex\hbox{c}}

\newcommand{\D}{\mathbb{D}}
\newcommand{\T}{\mathbb{T}}

\newcommand{\C}{\mathbb{C}}

\newcommand{\R}{\mathbb{R}}
\newcommand{\p}{\mathfrak{p}}

\newcommand{\q}{\mathfrak{q}}

\newtheorem{theorem}{Theorem}[section]

\newtheorem{lemma}[theorem]{Lemma}

\newtheorem*{theorem*}{Theorem}
\newtheorem*{conjecture*}{Conjecture}
\newtheorem{corollary}[theorem]{Corollary}
\newtheorem*{corollary*}{Corollary}

\newtheorem*{proposition*}{Proposition}

\def\bb{\begin{color}{blue}}
\def\bg{\begin{color}{green}}
\def\br{\begin{color}{red}}
\def\eg{\end{color}}
\def\er{\end{color}}
\def\eb{\end{color}}

\theoremstyle{remark}
\newtheorem{remark}[theorem]{Remark}
\newtheorem{question}{Question}
\newtheorem{definition}[theorem]{Definition}
\newtheorem{example}[theorem]{Example}

\author[Bickel]{Kelly Bickel}
\address{Department of Mathematics, Bucknell University, 360 Olin Science Building, Lewisburg, PA 17837, USA.}
\email{kelly.bickel@bucknell.edu}

\author[Pascoe]{James Eldred Pascoe}
\address{Department of Mathematics, University of Florida, 1400 Stadium Rd, Gainesville, FL 32611, USA.}
\email{pascoej@ufl.edu}

\author[Sola]{Alan Sola}
\address{Department of Mathematics, Stockholm University, Kr\"aftriket 6, 106 91 Stockholm, Sweden.}
\email{sola@math.su.se}
\thanks{Bickel was supported in part by National Science Foundation DMS grant $\#$1448846. Pascoe was supported in part by National Science Foundation Mathematical Science Postdoctoral Research Fellowship DMS $\#$1606260. Sola was supported in part by the Royal Swedish Academy of Sciences in the form of grant MG2018-0092 from Stiftelsen GS Magnusons fond. }

\begin{document}

\title[RIFs in higher dimensions]{Singularities of rational inner functions in higher dimensions}
\date{\today}

%\maketitle
%%%%%%%%%%%%%%%%%%%%%%%%%%%%%%%%%%%%%%%%%%%%%%

\begin{abstract} 
We study the boundary behavior of  rational inner functions (RIFs) in dimensions three and higher from both analytic and geometric viewpoints. On the analytic side, we use the critical integrability of the derivative of a rational inner function of several variables to quantify the behavior of a RIF near its singularities, and on the geometric side we show that the unimodular level sets of a RIF convey information about its set of singularities.
We then specialize to three-variable degree $(m,n,1)$ RIFs and conduct a detailed study of their derivative integrability, zero set and unimodular level set behavior, and non-tangential boundary values. Our results, coupled with constructions of non-trivial RIF examples, demonstrate that much of the nice behavior seen in the two-variable case is lost in higher dimensions. 
\end{abstract}
%%%%%%%%%%%%%%%%%%%%%%%%%%%%%%%%%%%%%%%%%%%%%%
\maketitle
%\tableofcontents

%%%%%%%%%%%%%%%%%%%%%%%%%%%%%%%%%%%%%%%%%%%%%%
\section{Introduction}\label{sec:intro}
\subsection{Singularities and critical integrability}
How singular is a function $h=g/f$ near a point where $f$ vanishes? There are many ways to make this question precise. One possibility, and the main focus of this paper, is to examine the integrability of different powers of $h$. Another possibility is to study how $h$ runs through different values as one approaches the singularity, for instance by analyzing the geometry of the level sets of $h$. This viewpoint also appears in the present work.

To make matters more concrete, suppose $h=1/f$, where $f \colon \mathbb{R}^n \to \mathbb{R}$ is a function with $f(0,\dots, 0)=0$. One classical approach to quantifying the behavior of $f$ near the origin is to determine the critical integrability index of $f$:
\begin{equation} \label{eqn:CII} \nu(f): = \sup \left \{ \epsilon >0 : \int_U |f(x)|^{-\epsilon} d\mu < \infty\right\},\end{equation}
where $U$ is some small set containing the origin and $\mu$ is a positive measure. The notion of critical integrability index arises naturally in several situations, for instance it has connections to the study of oscillatory integrals as in \cite{V76} as well as other applications in harmonic analysis, geometry, and PDE. See the brief discussion in \cite{CGP13} and the references therein. 

In complete generality, in dimensions higher than one and without any assumptions on the function $f$, the measure $\mu$, and the set $U$, it is a hard problem to analyze the critical integrability index. Typically, in applications, $\mu$ is Lebesgue measure and $U$ is an open set. Usually, it is also assumed  that $f$ exhibits at least some regularity near the origin. For instance, if $f$ is not smooth of finite type, then $\nu(f)$ may be equal to zero \cite{CGP13}, and even when $\nu(f)$ is positive, it can still be quite difficult to determine its exact value. See for example \cite{CCW99}, which considers critical integrability in a rather general setting, where only the existence of certain partial derivatives is assumed. 

There is a rich literature concerning the important case where $f$ is smooth and of finite type and $\mu$ is Lebesgue measure. We detail several recent developments that are particularly relevant to this paper. In the two-variable setting, $\nu(f)$ is closely related to the Newton distance $\Delta(f)$, which is defined using the Taylor series and subsequent Newton polygon of $f$; for precise definitions, see Section \ref{sec:3var1}.  In \cite{G06},  Greenblatt characterized the two-variable $f$ for which $\nu(f) = 1/\Delta(f)$ and studied the endpoint behavior. Earlier, in \cite{PSS99}, Phong, Stein, and Sturm considered the case of real-analytic two-variable functions $f$ and characterized their integrability indices using a family of Newton distances associated with $f$ defined via a family of analytic coordinate systems. Their work was generalized to the weighted setting by Pramanik in \cite{P02}. In higher dimensions, matters become more delicate and there are close connections with the general resolution of singularities \cite{Hau,G10}.  Recently, Collins, Greenleaf, and Pramanik developed a method for resolving singularities that allowed them to generalize the Phong-Stein-Sturm result to the $d$-variable case; specifically, they characterized $\nu(f)$ using numbers defined  via certain families of analytic coordinate systems associated to $f$ \cite{CGP13}.
For additional results related to the critical integrability index and its related circle of ideas, see for instance \cite{G10, CGP13, DHPT18} and the references therein. 

\subsection{Rational inner functions}
In this paper, we contribute to the theory by studying a notion of integrability index for an important 
class of bounded analytic functions of several complex variables, namely rational inner functions in $d$-dimensional polydisks. 
This is of course a restricted class of functions, but the additional algebraic structure of rational inner functions allows us to obtain significantly more information about integrability behavior than one could hope to obtain in the general situation. At the same time, as is explained below, rational inner functions play a very significant role in multivariate function  and  operator theory, providing us with ample motivation to study their behavior near singularities.

The unit polydisk in $\mathbb{C}^d$ is the set  
\[\mathbb{D}^d =\{ (z_1, \dots, z_d) : |z_j|<1 \text{ for } j=1, \dots, d\}.\] 
A ratio of  $d$-variable polynomials $\phi$ is called a {\it rational inner function}, or RIF, if it is holomorphic on $\mathbb{D}^d$ and if 
\[ \lim_{r \nearrow 1} | \phi(r \zeta) |  =1 \text{ for a.e. } \zeta \in \mathbb{T}^d : =\{ (\zeta_1, \dots, \zeta_d) : |\zeta_j|=1 \text{ for } j=1, \dots, d\}.\]
Rational inner functions enjoy certain structural properties not directly apparent in the definition. 
Rudin and Stout \cite{RudStout65} showed that $\phi$ is a RIF if and only if 
 \[ \phi(z) = \lambda \frac{ \tilde{p}(z)}{p(z)},\]
 where $p$ is a polynomial with no zeros on $\mathbb{D}^d$, $\tilde{p}$ is the reflection of $p$, and $\lambda \in \mathbb{T}.$ To define $\tilde{p},$ let $\deg \phi = (n_1, \dots, n_d)$ denote the maximum powers of $z_1, \dots, z_d$ in $\phi$. Then 
 \[ \tilde{p}(z) := z_1^{n_1} \dots z_d^{n_d} \overline{ p( \tfrac{1}{\bar{z}_1}, \dots, \tfrac{1}{\bar{z}_d})}.\]
 Moreover, one can assume that $p$ is atoral \cite{AMS06}. For the purposes of this paper, atoral implies  that $p$ and $\tilde{p}$ share no common factors and if $\mathcal{Z}_p$ denotes the zero set of $p$, then $\dim (\mathcal{Z}_p \cap \mathbb{T}^d) \le d-2.$ In what follows, we will typically assume, without loss of generality, that the unimodular constant $\lambda =1$.
 
The structured class of rational inner functions plays a key role in the study of holomorphic functions on the unit polydisk. Indeed, RIFs are the $d$-variable generalizations of finite Blaschke products \cite{GMRBook}. As such, it should not be surprising that every holomorphic $f\colon  \mathbb{D}^d \rightarrow \mathbb{D}$, that is, every Schur function, can be approximated locally uniformly by RIFs \cite{Rud69}. This has, for example, been used to give proofs that every Schur function on the bidisk has an important structural feature, called an Agler decomposition \cite{B12, Kne08, W10}.
 Rational inner functions also appear as solutions to Nevanlinna-Pick interpolation problems and are used to generate key examples of functions that preserve matrix inequalities \cite{AM02, AMY12}. Denominators of rational inner functions, termed stable polynomials, also make appearances in other parts of analysis, for instance in dynamical systems \cite[Example 4.3]{LSV13}.
Both rational inner functions and stable polynomials have close connections to engineering applications. For example, rational inner functions on $\mathbb{D}^d$ serve as the transfer functions for $d$-dimensional, dissipative, linear, discrete-time input-state-systems with finite-dimensional state spaces, see \cite{BSV05, BK16}. Similarly, rational inner functions have applications to multidimensional lossless networks, which in turn are related to multidimensional wave digital filters \cite{Kum02}.
 
Although they generalize finite Blaschke products, RIFs are much more complicated than their one-variable counterparts.  Most importantly, unlike finite Blaschke products, they can have singularities on the boundary of $\mathbb{D}^d$, which can take several forms. For instance, consider
 \[ \phi(z) : = \frac{3z_1z_2z_3-z_1 -z_2-z_3}{3-z_1-z_2-z_3} \ \ \ \text{ and } \ \ \ \psi(z) :=\frac{3z_1z_2z_3 -z_3-z_1z_2}{3-z_1z_2-z_3}.\]
Then if $\phi =\frac{\tilde{p}}{p}$ and $\psi = \frac{\tilde{q}}{q}$, we can see that $\mathcal{Z}_p \cap \mathbb{T}^3 = \{ (1,1,1)\}$ and $\mathcal{Z}_q \cap \mathbb{T}^3 = \{ (e^{i\theta}, e^{-i\theta}, 1): \theta \in [0, 2\pi)\}$. Morally, this illustrates the two possibilities in $3$-variables because $\dim \left( \mathcal{Z}_p \cap \mathbb{T}^d \right) \le d-2$ in general, and hence, $\mathcal{Z}_p \cap \mathbb{T}^3$ should heuristically be composed of points and curves.
 
In this paper, we study the following concrete version of the general question alluded to earlier:
 \begin{equation} \label{eqn:Q1} \text{ How ``singular'' can a RIF on $\mathbb{D}^d$ be near its singularities on $\mathbb{T}^d$?} \end{equation}
 Some results are known, especially in two dimensions. For example, Knese showed in \cite[Corollary 14.6]{Kne15}  that every RIF has a nontangential 
 boundary value at every $\tau \in \mathbb{T}^d$, including its singular points. Stronger notions of nontangential regularity (basically, nontangential polynomial approximation) for two-variable RIFs were  
 studied by M\McC Carthy and Pascoe in \cite{MP17}.  Similarly, Knese \cite{Kne15} studied integral behavior of stable polynomials $p \in \mathbb{C}[z_1, z_2]$, i.e. the denominators of RIFs on $\mathbb{D}^2$, and in particular, characterized the $q \in \mathbb{C}[z_1,z_2]$ for which $p/q \in L^2(\mathbb{T}^2).$ 

 In this paper, we  approach \eqref{eqn:Q1} from several angles; we both conduct an analytic study of the critical integrability indices of partial derivatives $\frac{\partial \phi}{\partial z_1}, \dots, \frac{\partial \phi}{\partial z_d}$ near singularities on $\mathbb{T}^d$ and also conduct a  geometric study of the structure of a RIF's unimodular level sets near the singularities on $\T^d$. 
 
The inherent interest and importance of rational inner functions, and the remarkable intricacy of their boundary behavior, serves as the primary motivation for our present study. One additional reason to pursue this study is that integrability results could provide an invariant that differentiates between the boundary behavior of RIFs and that of locally inner functions. This could then lead to a possible obstruction to the open question of whether locally inner functions can be approximated by RIFs on $\mathbb{T}^d$ in a tractable way. This question is connected to the work of Agler, M\McC Carthy, and Young in \cite{AMY12}, where they characterize two-variable locally matrix monotone functions. For two-variable rational inner functions, their local results imply global matrix monotonicity. Boundary approximation via rational inner functions has been suggested as a possible way to extend their local results to more general characterizations of global matrix monotonicity.

 \subsection{Two-variable RIFs.} In \cite{BPS18, BPS19}, the authors conducted an in-depth study of \eqref{eqn:Q1} for RIFs $\phi$ on $\mathbb{D}^2.$ Let us briefly discuss the most salient results, as they will inform our study in higher dimensions. The investigations revolved around two geometric objects associated to $\phi$:
 \begin{itemize}
\item[(1)] The zero set of $\phi = \frac{\tilde{p}}{p}$, i.e. $\mathcal{Z}_{\tilde{p}}$, restricted to the faces of the bidisk: $(\mathbb{T} \times\mathbb{D} )\cup (\mathbb{D} \times\mathbb{T} )$.
\item[(2)]  The unimodular level sets of $\phi$ on $\mathbb{T}^2$; i.e.~for $\lambda \in \mathbb{T}$,  the closure of the set $\{ (\zeta_1, \zeta_2) \in \mathbb{T}^2: \phi(\zeta)= \lambda \}$, which we denote by $\mathcal{C}_{\lambda}.$
\end{itemize}
In  \cite{BPS18}, we focused on (1), namely the geometry of $\mathcal{Z}_{\tilde{p}}$ on the face of the bidisk near a singular point $\tau = (\tau_1, \tau_2) \in \mathbb{T}^2$. To restrict to $\mathbb{D} \times\mathbb{T}$, fix $\zeta_2 \in \mathbb{T}$ near $\tau_2$ and let $\alpha_1(\zeta_2), \dots, \alpha_m(\zeta_2)$ denote the points in $\mathbb{D}$ where $\tilde{p}(\alpha_i(\zeta_2), \zeta_2)=0.$ We showed that  $\mathcal{Z}_{\tilde{p}} \cap (\mathbb{D} \times \mathbb{T})$ approaches $\tau$ in the following way; there is a positive, even integer $K_1$ so that 
\[ \min_{1 \le i \le m}\left( 1-|\alpha_i(\zeta_2)| \right) \approx | \tau_2 -\zeta_2|^{K_1},\]
for all $\zeta_2 \in \mathbb{T}$ sufficiently close to $\tau_2$.  This $K_1$ is called the \emph{$z_1$-contact order of $\phi$ at $\tau$} and, after taking the maximum $K_1$ over all the singularities of $\phi$, it characterizes the critical integrability index of $\frac{\partial \phi}{\partial z_1}$; for $1 \le \p < \infty$, 
\[ \tfrac{\partial \phi}{\partial z_1} \in L^\p(\mathbb{T}^2) \ \  \text{ if and only if }  \ \ K_1  < \tfrac{1}{\p-1}.\]
We moreover derived a surprising inverse relationship between better non-tangential regularity of $\phi$ (i.e. non-tangential approximation by a higher degree polynomial) and higher derivative integrability (i.e. smaller contact orders). See Theorems $3.3$ and $4.1$ and Corollary $7.2$ in \cite{BPS18}.

In \cite{BPS19}, we focused on (2), the geometry of the $\mathcal{C}_{\lambda}$ on $\mathbb{T}^2$ near a singular point. As the singular points in dimension two are isolated, it is easy to show that each $\mathcal{C}_{\lambda}$ contains $\mathcal{Z}_{\tilde{p}} \cap \mathbb{T}^2$ in its closure.  Then it is reasonable to study how the $\mathcal{C}_{\lambda}$ approach singular points $\tau$. 
Using properties of RIFs and Puiseux series expansions, we showed that near each singular point $\tau \in \mathbb{T}^2$, each $\mathcal{C}_{\lambda}$ can be parameterized by analytic functions
\[  z_1 = \psi^{\lambda}_{1}(z_2), \dots, z_1 =  \psi^{\lambda}_{L}(z_2)\]
centered at $\tau_2$. Then another way to define the $z_1$-contact order of $\phi$ at $\tau$ is the following;
  $K_1$ is the maximal order of vanishing of
$\psi^{\lambda}_i-\psi^{\mu}_j$ at $\tau_2$ for two generic $\mu,\lambda \in \mathbb{T}$. This result made it possible to use pictures of the unimodular level curves to 
observe quantitative integrability facts about RIFs. It also allowed us to link the $z_1$- and $z_2$-contact orders of $\phi$ and subsequently conclude that 
\[ \tfrac{\partial \phi}{\partial z_1} \in L^\p(\mathbb{T}^2) \ \  \text{ if and only if } \ \  \tfrac{\partial \phi}{\partial z_2} \in L^\p(\mathbb{T}^2).\]
Namely, the two partial derivatives of $\phi$ must have the same critical integrability indices. We also conducted a finer analysis of the geometry of $\mathcal{C}_\lambda$ but will not discuss that further here. See Theorems $2.8$, $3.1$, and $4.1$ in \cite{BPS19} for more details.
\subsection{Main Results}

In this paper, we explore similar questions for $d$-variable RIFs $\phi = \frac{\tilde{p}}{p}$. Naively, one might expect results quite similar to the two-dimensional case, perhaps with level curves replaced by level hypersurfaces and with methods only slightly modified because of additional multi-indices. This turns out not to be the case, and one of the main themes of this paper is that many of the nice features uncovered in \cite{BPS18, BPS19} are absent in higher dimensions.

In several complex variables, the move from one to two variables is often easier than going from two to three or more variables. For instance, B\'ezout's theorem \cite{FulBook} and And\^o's inequality \cite{AM02} are key two-variable results that provide useful generalizations of one-variable results but lack tractable three-variable counterparts. An example of a more specialized issue is that every two variable RIF has a unitary transfer function realization but this is no longer the case in $d\geq 3$ variables. Moreover, even when such a realization does hold, it need not be as minimal as one might expect, see \cite{Kne11Publ}. 

Our study provides further examples of this general higher-dimensional phenomenon. Here are two specific examples of difficulties that arise when increasing the dimension to $d\geq 3$:\\
\begin{itemize}
\item[i.] The singular set $S = \mathcal{Z}_p \cap \mathbb{T}^d$ becomes much more complicated. It satisfies $\dim S \le d-2$, but the precise dimension and overall structure can vary by RIF and along components. Additionally, $\mathcal{Z}_p$ can no longer be described using Puiseux series and even when we can locally describe $\mathcal{Z}_p$ as $z_3 = \psi^0(z_1, z_2)$, the function $\psi^0$ can be very discontinuous. \\

\item[ii.] The unimodular level sets $\mathcal{C}_{\lambda}$ cannot generally be parameterized by analytic functions. Indeed, for many three-variable RIFs $\phi$, the $\mathcal{C}_{\lambda}$ will have discontinuities at points on $\mathcal{Z}_p \cap \mathbb{T}^3.$ In two-variables, pictures of unimodular level curves conveyed quantitative integrability information about $\phi$, but already in three variables, the pictures are much more complicated.\\
\end{itemize}

\noindent  Not only do our methods become less effective or even inapplicable, but some of our key two-variable results actually fail in higher dimensions, see Examples \ref{ex:vl1} and \ref{ex:curve2}: in general, different partials $\frac{\partial \phi}{\partial z_j}$ of a RIF exhibit different critical integrability, and as mentioned above, unimodular level sets need not admit a smooth or even continuous parametrization. Because of this reality, our paper has two complementary goals: \\

\begin{itemize}
\item[Goal 1.] Establish results for $d$-variable RIFs concerning their singular sets $\mathcal{Z}_p$, their unimodular level sets $\mathcal{C}_{\lambda}$, and the integrability properties of their partial derivatives on $\mathbb{T}^d$.\\

\item[Goal 2.] Produce examples illustrating the various complexities that arise in the $d$-variable case. \\
\end{itemize}

\noindent We now summarize our main findings. In Section \ref{sec:gen}, we tackle Goal $1$ and establish two important facts about $d$-variable rational inner functions. First, in Theorem \ref{thm:GenInt}, we 
characterize the integrability of $\frac{\partial \phi}{\partial z_d} $ in terms of how $\mathcal{Z}_{\tilde{p}} \cap (\mathbb{T}^{d-1} \times \mathbb{D})$ approaches $\mathbb{T}^{d}$.
Specifically, for each $\hat{\zeta} \in \mathbb{T}^{d-1}$, let $\delta(\phi, \hat{\zeta}) $ denote the distance between $\mathcal{Z}_{\tilde{p}} \cap (\{\hat{\zeta}\} \times \mathbb{D})$ and $\mathbb{T}^{d}$. For
each $ x >0$, define
\[ \Omega_x = \left \{ \hat{\zeta} \in \mathbb{T}^{d-1}: \delta(\phi, \hat{\zeta}) < \tfrac{1}{x} \right\}.\]
Then, letting $\mu$ denote Lebesgue measure, we establish:\\

\noindent \textbf{Theorem \ref{thm:GenInt}.} \emph{For $1\le \p < \infty$, $\frac{\partial \phi}{\partial z_d} \in L^\p(\mathbb{T}^d)$ if and only if 
$\displaystyle \int_1^{\infty} \mu \left( \Omega_x \right)  \ x^{\p-2} \ dx < \infty.$} \\
It should be noted that the argument that leads to \cite[Proposition 4.4]{BPS18} extends to the $d$-variable setting: for any RIF and any index $j \in \{1,\ldots,d\}$,
\[\left\|\tfrac{\partial  \phi}{\partial z_j}\right\|_{L^1(\mathbb{T}^d)}=n_j<\infty,\]
where $n_j$ is the $j$-degree of $\phi$. Thus, there is no loss in assuming $\p\geq 1$ throughout.
In  the two-variable setting, Theorem \ref{thm:GenInt} combined with the definition of contact order from \cite{BPS18} gives the integrability results from \cite{BPS18}. See Remark \ref{rem:2var} for details.
In Section \ref{sec:gen}, we also study the relationship between the singular set  $\mathcal{Z}_p \cap \mathbb{T}^d$ and the unimodular level sets $\mathcal{C}_{\lambda}$ of $\phi$. In particular, we prove that every unimodular level set of $\phi$ on $\mathbb{T}^d$ goes though the singular set of $\phi$. This demonstrates that the unimodular level sets are a viable tool for studying ``how singular'' a RIF $\phi$ is near its singular set on $\mathbb{T}^d$.

In Sections  \ref{sec:3var1} and  \ref{sec:3var2}, we restrict to irreducible degree $(m,n,1)$ RIFs on $\mathbb{D}^3$ with singularities on $\mathbb{T}^3$. This enables us to sidestep certain technical difficulties and allows us to perform a finer analysis on their integrability, zero set, and unimodular level set behaviors. For such RIF $\phi = \frac{\tilde{p}}{p}$, Section \ref{sec:3var1} examines the integrability of $\frac{\partial \phi}{\partial z_3}$. First in \eqref{eqn:psi}, we point out that there is a function $\rho_{\phi}$ such that 
\begin{equation} \label{intro:eqn}  \int_{\mathbb{T}^3} \Big | \tfrac{\partial \phi}{\partial z_3} (\zeta ) \Big|^\p |d\zeta | 
\approx \iint_{[-\pi,\pi]^2} \left | \rho_{\phi}(\theta_1, \theta_2) \right |^{1-\p} d\theta_1 d\theta_2. \end{equation}
If $\mathcal{Z}_{p} \cap \mathbb{T}^3$ does not contain any vertical lines of the form $\{(\zeta_1, \zeta_2)\} \times \mathbb{T}$, then $\rho_{\phi}$ is real analytic. This puts the question of when \eqref{intro:eqn} is finite in the setting of Greenblatt's results from \cite{G06}. The required definitions and results from  \cite{G06} are detailed in Subsection \ref{subsec:GT}, particularly in Theorem \ref{thm:Green}. Then in Subsection \ref{subsec:int}, we deduce properties about the Taylor series expansions of the $\rho_{\phi}$ associated to our RIFs. Combining these results with Theorem \ref{thm:Green} allows us to deduce some integrability properties of RIFs, for example:\\

\noindent \textbf{Corollary \ref{cor:int}}.  \emph{Assume $\phi = \frac{\tilde{p}}{p}$ is a singular degree $(m,n,1)$ irreducible RIF and $\mathcal{Z}_{\tilde{p}}\cap \mathbb{T}^3$ does not contain any vertical lines $\{(\zeta_1, \zeta_2)\} \times \mathbb{T}$. Then 
\begin{itemize}
\item[i.] If $\dim\left(\mathcal{Z}_p \cap \mathbb{T}^3\right)=0$, then $\frac{\partial \phi}{\partial z_3} \not \in L^\p(\mathbb{T}^3)$ for $\p \ge 2.$
\item[ii.] If $\dim\left(\mathcal{Z}_p \cap \mathbb{T}^3\right)=1$, then $\frac{\partial \phi}{\partial z_3} \not \in L^\p(\mathbb{T}^3)$  for $\p \ge \alpha,$ for some $\alpha <2.$\
\end{itemize}}

However, there are significant limitations to these methods. For example, RIFs whose singular sets contain vertical lines can have discontinuous $\rho_{\phi}$, see Example \ref{ex:vl2}. Then Theorem \ref{thm:Green} does not apply to them. Similarly, Theorem \ref{thm:Green} only says when $1/\Delta(\rho_{\phi})$ will give the integrability index, where $\Delta(\rho_{\phi})$ is the Newton distance of $\rho_{\phi}$. For some RIFs, $1/\Delta(\rho_{\phi})$ will not yield the integrability index of $\rho_{\phi}$ and so, other methods or a direct analysis will be required. See for instance Examples \ref{ex:lifted} and \ref{ex:curveiso}. 

 Section \ref{sec:3var2} examines the boundary values and unimodular level sets of irreducible degree $(m,n,1)$ RIFs, refining the results in \cite{Kne15}. In particular, in Theorem \ref{thm:lim}, we study the nontangential boundary values of such $\phi$ on $\mathbb{T}^d$, with an emphasis on points in $\mathcal{Z}_p \cap \mathbb{T}^3$; surprisingly, the boundary values exhibit different behavior depending on whether $\mathcal{Z}_p \cap \mathbb{T}^3$ contains a vertical line or not. We then study the unimodular level sets $\mathcal{C}_{\lambda}$ and as part of Theorem \ref{lem:RIF} prove: \\
 
 \noindent \textbf{Theorem \ref{lem:RIF}* }\emph{ Given any $\lambda \in \mathbb{T}$, the unimodular level set $\mathcal{C}_{\lambda}$ is composed of a finite number of vertical lines $\{(\zeta_1, \zeta_2)\} \times \mathbb{T}$ and a surface of the form:
\[ \tau_3= \Psi_{\lambda}(\tau_1, \tau_2)  \ \ \text{ for } \ \ (\tau_1, \tau_2)\in \mathbb{T}^2,\]
where $1/\Psi_{\lambda}$ is a two-variable RIF. \\
}

In Section \ref{sec:higherzoo} and throughout the paper, we also tackle Goal 2. Specifically, we illustrate theorems and disprove a number of potential conjectures using nontrivial RIF examples. Here is a selection of important examples and some of the information that they convey.\\

\begin{enumerate}
\item Example \ref{ex:vl1} provides a singular, irreducible degree $(1,1,1)$ RIF $\phi$ such that $\frac{\partial \phi}{\partial z_3} \in L^\p(\mathbb{T}^3)$ for all $1 \le \p <\infty$ but $\frac{\partial \phi}{\partial z_1},$ $\frac{\partial \phi}{\partial z_2}  \in  L^\p(\mathbb{T}^3)$ if and only if $\p < \frac{3}{2}.$ This shows that in three variables, the integrability indices for partial derivatives are not necessarily equal. \\

\item Example \ref{ex:curve} provides a degree $(2,1,1)$ RIF whose generic unimodular level sets each have two singular points. This shows that in three variables, the unimodular level sets associated to RIFs need not be smooth.\\

\item Examples \ref{ex:iso1}, \ref{ex:curve}, \ref{ex:curve2}, \ref{ex:curveiso},  and \ref{ex:curveiso2}  illustrate the possible forms that $\mathcal{Z}_p \cap \mathbb{T}^3$ can take and their interplay with integrability.  In particular, these examples possess an isolated singularity, a curve of singularities, a curve of singularities, a combination of the two, and an isolated singularity, respectively. While the first $\frac{\partial \phi}{\partial z_3}  \in L^\p(\mathbb{T}^3)$ if and only if $\p<2$, in Examples \ref{ex:curve}, \ref{ex:curveiso} and \ref{ex:curveiso2} we have $ \frac{\partial \phi}{\partial z_3} \in L^\p(\mathbb{T}^3)$ if and only if $\p<\frac{3}{2}$, and in Example \ref{ex:curve2}, the integrability range (for $\frac{\partial \phi}{\partial z_2}$) is $\p<\frac{5}{4}$. \\

\end{enumerate}
\noindent These examples indicate that there appears to be no easy way to characterize integrability in three or more variables just by studying the basic properties of the singular set $\mathcal{Z}_p$ and the unimodular level sets $\mathcal{C}_{\lambda}$.

\section{General RIFs} \label{sec:gen}
\subsection{Integrability.} 
Let $\phi = \frac{\tilde{p}}{p}$ be a general rational inner function on $\mathbb{D}^d$ with $\deg \phi = (n_1, \dots, n_d)$, and let $S = \mathcal{Z}_{p} \cap \mathbb{T}^d$ denote its singular set on $\mathbb{T}^d$. Note that, by the definition of $\tilde{p}$, we also 
have $S=\mathcal{Z}_{\tilde{p}}\cap\mathbb{T}^d$. In this section, we characterize the critical integrability index for the partial derivatives $ \tfrac{\partial \phi}{\partial z_i}$ using properties of $\mathcal{Z}_{\tilde{p}}$. First, without loss of generality, restrict to the single partial derivative $  \tfrac{\partial \phi}{\partial z_d}$.  Then, the integrability of this partial derivative will be governed by how $\mathcal{Z}_{\tilde{p}} \cap (\mathbb{T}^{d-1} \times \mathbb{D})$ approaches $S$.

To make this precise, we will associate $\phi$ to a family of finite Blaschke products parameterized by most $\hat{\zeta} \in \mathbb{T}^{d-1}$. To define the exceptional set, let $\pi(S)$ denote the projection of $S$ onto $\mathbb{T}^{d-1}$:
\[ \pi(S) = \left \{ \hat{\zeta} \in \mathbb{T}^{d-1}: \tilde{p}(\hat{\zeta}, \zeta_d) = 0 \text{ for some } \zeta_d \in \mathbb{T} \right \}.\]
Since $\dim S \le d-2$, $\dim \pi(S) \le d-2$ and so, has Lebesgue measure zero on $\mathbb{T}^{d-1}$.
Fix $\hat{\zeta} \in \mathbb{T}^{d-1} \setminus \pi(S)$ and define the sliced function
\[\phi_{\hat{\zeta}}(z_d) := \phi(\hat{\zeta}, z_d) = \frac{\tilde{p}(\hat{\zeta}, z_d)}{p(\hat{\zeta}, z_d)}.\]
Let $\{\hat{z}_n\} \subseteq \mathbb{D}^{d-1}$ denote a sequence converging to 
$\hat{\zeta}$.  Then each one-variable polynomial $p_{\hat{z}_n}(z):=p(\hat{z}_n,z)$ is nonvanishing on $\mathbb{D}$ and so, Hurwitz's theorem implies that 
$p_{\hat{\zeta}}$ is either nonvanishing on $\mathbb{D}$ or identically zero. Since $\hat{\zeta} \not \in \pi(S)$, the limit polynomial $p_{\hat{\zeta}}$ must be nonvanishing on $\mathbb{D}$. 
Moreover, 
\[ | \tilde{p}(\hat{\zeta}, \zeta_d)| = | {p(\hat{\zeta}, \zeta_d)}| \ne 0\]
for all $\zeta_d \in \mathbb{T}$. This implies that $\phi_{\hat{\zeta}}$ 
is a finite Blaschke product with $n_{\hat{\zeta}} :=\deg \phi_{\hat{\zeta}}  \le  n_d.$ To study $\mathcal{Z}_{\tilde{p}}$, let $\alpha_1, \dots, \alpha_{n_{\hat{\zeta}}}\in \mathbb{D}$ denote the zeros of $\phi_{\hat{\zeta}}$.  Define the minimal distance of these zeros from $\mathbb{T}$ by 
\[ \delta(\phi, \hat{\zeta}) = \min_{1 \le i \le n_{\hat{\zeta}}} \big( 1 - |\alpha_i| \big).\]
Note that $\delta(\phi, \hat{\zeta})$ is measuring the distance 
of $\mathcal{Z}_{\tilde{p}} \cap (\{\hat{\zeta}\} \times \mathbb{D})$ to $\mathbb{T}^{d}$. For
each $ x >0$, define
\[ \Omega_x = \left \{ \hat{\zeta} \in \mathbb{T}^{d-1}: \delta(\phi, \hat{\zeta}) < \tfrac{1}{x} \right\}.\]
Then we can use this to control the derivative integrability of $\frac{\partial \phi}{\partial z_d}$ as follows:

\begin{theorem} \label{thm:GenInt} Let $\phi = \frac{\tilde{p}}{p}$ be a RIF on $\mathbb{D}^d$. Then for $ 1\le \p < \infty$, $\frac{\partial \phi}{\partial z_d} \in L^\p(\mathbb{T}^d)$ if and only if 
\[ \int_1^{\infty} \mu \left( \Omega_x \right)  \ x^{\p-2} \ dx < \infty.\]
\end{theorem}

To prove this, we require the following lemma. It is likely well known, but also easily follows from the arguments in Lemmas $4.2$ and $4.3$ in \cite{BPS18}. 

\begin{lemma} \label{lem:fbp} Let $b$ be a finite Blaschke product with zeros $\alpha_1, \dots, \alpha_{n} \in \mathbb{D}$ and define $\delta(b) = \min_{1 \le i \le n} (1 - |\alpha_i|).$
Then for $1 \le \p < \infty$
\[\int_{\mathbb{T}} |b'(z)|^\p \ |dz| \approx  |\delta(b)|^{1-\p},\]
where the implied constant depends on $n$ and $\p$.
\end{lemma}
%Here is the proof of Theorem \ref{thm:GenInt}:
\begin{proof}[Proof of Theorem \ref{thm:GenInt}] Since $\pi(S) \times \mathbb{T}$ has Lebesgue measure $0$ on $\T^d$, we have
\[ \begin{aligned}
\int_{\mathbb{T}^d} \left | \tfrac{\partial \phi}{\partial z_d} (\zeta ) \right|^\p \ | d\zeta | 
& = \int_{(\mathbb{T}^{d-1} \setminus \pi(S)) \times \mathbb{T}} \left | \tfrac{\partial \phi}{\partial z_d} (\zeta ) \right|^\p \ | d\zeta | \\
& =  \int_{\mathbb{T}^{d-1}\setminus \pi(S)} \int_{\mathbb{T}}    \left | \phi_{\hat{\zeta}}'(\zeta_d) \right|^\p \ |d\zeta_d|  | d \hat{\zeta} |  \\
&\approx  \int_{\mathbb{T}^{d-1}\setminus \pi(S)} \left |  \delta(\phi, \hat{\zeta})  \right |^{1-\p} \ | d \hat{\zeta} |  \\
& = \int_0^{\infty} \mu\left( \hat{\zeta} \in \mathbb{T}^{d-1} \setminus \pi(S): \delta(\phi, \hat{\zeta})^{-1} >x \right) \ x^{\p-2} \ dx \\
& = \int_0^{\infty} \mu\left( \hat{\zeta} \in \mathbb{T}^{d-1} \setminus \pi(S): \delta(\phi, \hat{\zeta}) < \tfrac{1}{x} \right) \ x^{\p-2} \ dx \\
& = \int_0^{\infty} \mu(\Omega_x ) \ x^{\p-2} \ dx \\
& \approx \int_1^{\infty} \mu(\Omega_x) \ x^{\p-2} \ dx,
\end{aligned}
\]
where we used Lemma \ref{lem:fbp} and the fact that each $\mu(\Omega_x) \le (2\pi)^{d-1}$.
\end{proof}

\begin{remark} \label{rem:2var} Theorem \ref{thm:GenInt} combined with the definition of contact order can be used to derive the two-variable integrability result from \cite{BPS18}. 
To see this, let $\phi = \frac{\tilde{p}}{p}$ be a two-variable RIF. Then, the singular set $\mathcal{Z}_{\tilde{p}} \cap \mathbb{T}^2$ is finite. To simplify notation (without changing the idea of the proof), assume  $\mathcal{Z}_{\tilde{p}} \cap \mathbb{T}^2 = \{(1,1)\},$ i.e. $\phi$ has a single singular point on $\mathbb{T}^2$ at $(1,1).$
Then by Theorem $3.3$ in \cite{BPS18}, there is a positive, even integer $K_1$, called the $z_1$-contact order of $\phi$, so that
\begin{equation} \label{eqn:CO}  \delta(\phi, \zeta_2) \approx |\zeta_2 - 1 |^{K_1},\end{equation}
for all $\zeta_2 \in \mathbb{T}$ sufficiently close to $1$, say in some neighborhood $V_{\epsilon}(1) \subseteq \mathbb{T}.$ It is easy to show that  there is some $c>0$ so that  for all $\zeta_2 \in \mathbb{T} \setminus V_{\epsilon}(1)$,
$\delta(\phi, \zeta_2) \ge c >0$. This implies that 
\[ \int_1^{\infty} \mu(\Omega_x) \ x^{\p-2} \ dx <\infty \text{ if and only if } \int_1^{\infty} \mu(\Omega_x \cap V_{\epsilon}(1)) \ x^{\p-2} \ dx <\infty.\]
A simple application of \eqref{eqn:CO} shows that for $x$ sufficiently large,
\[  \mu(\Omega_x \cap V_{\epsilon}(1)) \approx x^{-\tfrac{1}{K_1}}.\]
This immediately implies that 
\[  \int_1^{\infty} \mu(\Omega_x \cap V_{\epsilon}(1)) \ x^{\p-2} \ dx <\infty \text{ if and only if }  \int_1^{\infty}x^{\p-2-\tfrac{1}{K_1}} \ dx <\infty.\]
The last integral is finite if and only if $K_1 < \frac{1}{\p-1}$, which is exactly the characterization of $\frac{\partial \phi}{\partial z_1} \in L^\p(\mathbb{T}^2)$ that was proved in \cite{BPS18}. So, in the absence of a quantity like contact order, a result like Theorem \ref{thm:GenInt} is the best one might expect.
\end{remark}

\begin{remark}\label{rem:isovscurve}As is explained in Remark \ref{rem:2var}, a critical integrability index $1+\frac{1}{K_1}$ for $\frac{\partial \phi}{\partial z_1}$ (and, by extension, for $\frac{\partial \phi}{\partial z_2}$) can only be obtained by having at least one singular point with contact order $K_1$; note also that $K_1$ is always even. In higher dimensions,  the situation is more complicated and one could imagine realizing a particular rate of decay in $\mu(\Omega_x)$ in different ways.  Indeed, as we will see, there are RIFs with isolated singularities whose partial derivatives exhibit the same critical integrability as RIFs with curve singularities. See Examples \ref{ex:curve} and \ref{ex:curveiso2}.
\end{remark}
In practice, Theorem \ref{thm:GenInt} can be used to deduce the integrability properties of simple RIFs. Consider the following canonical example, which also appears in \cite{Kne11Ill}:

\begin{example} \label{ex:can} Define the RIF $\phi_d\colon \D^d\to \C$ by
\[\phi_d(z)=\frac{d\prod_{k=1}^dz_k-\sum_{j \in J}z_{j_1}\cdots z_{j_{d-1}}}{d-\sum_{k=1}^dz_k},\]
where $J=\{(j_1,\ldots, j_{d-1})\in \mathbb{N}^{d-1}\colon 1 \le j_1 < j_2 < \dots < j_{d-1} \le d\}.$
Each $\phi_d$ has a singularity at $(1,\ldots,1) \in \T^d$ and is smooth on $\T^d\setminus\{(1,\ldots,1)\}$. Using Theorem \ref{thm:GenInt}, we
can prove:
\[ \tfrac{\partial \phi_d}{\partial z_k}\in L^{\p}(\T^d) \ \  \text{ if and only if } \p<\tfrac{1}{2}(d+1). \]
By symmetry, it is enough to prove the result for $\frac{\partial \phi_d}{\partial z_d}$. Let $\vec{1} =(1,\dots, 1) \in \mathbb{T}^{d-1}$. Then by Theorem 
\ref{thm:GenInt} it suffices to show that for some $\epsilon >0$ and neighborhood $V_{\epsilon}(\vec{1})$ of $\vec{1}$ in $\mathbb{T}^{d-1}$,
 \[ \int_1^{\infty} \mu(\Omega_x \cap V_{\epsilon}(\vec{1})) \ x^{\p-2} \ dx < \infty \ \ \text{ if and only if } \ \ \p<\tfrac{1}{2}(d+1).\]
We first need to understand $\delta(\phi, \hat{\zeta}).$  Solving $\tilde{p}(z)=0$ for $z_d$, we find that 
\[z_d=\frac{\prod_{k=1}^{d-1}z_k}{d\prod_{k=1}^{d-1}z_k-\sum_{j \in \check{J}}z_{j_1}\cdots z_{j_{d-2}}}.\]
Here, $\check{J}=\{(j_1,\ldots, j_{d-2}) \in \mathbb{N}^{d-2}\colon 1 \le j_1 < j_2 < \dots < j_{d-2} \le d-1\}$.
Evaluating for $\hat{\zeta} = (\zeta_1, \dots, \zeta_{d-1})\in \T^{d-1}$, we have
\[\frac{1}{\left|z_d(\hat{\zeta})\right|^2}=d^2-2d\sum_{k=1}^{d-1}\mathrm{Re}(\zeta_k)+\sum_{j,k=1}^ {d-1}\zeta_j\overline{\zeta}_k\approx 1+2\sum_{k=1}^{d-1}\theta_k^2+2\sum_{1\leq j<k\leq d-1}\theta_j\theta_k+\mathcal{O}(\|\theta\|^3)\]
for $\theta \in [-\pi, \pi]^{d-1}$ near $(0,\dots, 0)$. Hence,
\[ \delta (\phi, \hat{\zeta})\approx 1-|z_d(\hat{\zeta})|^2 = 2\left(\sum_{k=1}^{d-1}\theta_k^2+\sum_{1\leq j<k \leq d-1}\theta_j\theta_k\right)+\mathcal{O}(\|\theta\|^3),\]
for $\theta \in [-\pi, \pi]^{d-1}$ near $(0,\dots, 0)$. 
The expression in parentheses is a real quadratic form whose associated symmetric matrix $M=(m_{j,k})_{j,k=1}^{d-1}$ satisfies
\[m_{j,k}=\left\{\begin{array}{cc}1 & \textrm{if}\,\, j=k\\
\frac{1}{2} &\textrm{otherwise}\end{array}\right..\]
Then $M$ is a symmetric circulant matrix whose eigenvalues are given by
\[\lambda_{\ell}=1+\tfrac{1}{2}\sum_{k=1}^{d-1}\omega_{\ell}^k,\]
where $\omega_{0}, \ldots, \omega_{d-1}$ are the $d$:th roots of unity (viz. \cite[Theorem 12.5.7]{GHBook}). Using the fact that $\sum_{k=0}^{d-1}\omega_{\ell}^k=0$ for $1\leq \ell \leq d-1$, we can deduce that $M$ has one eigenvalue equal to $\frac{d+1}{2}$ and $d-2$ eigenvalues equal to $\frac{1}{2}$.
Then by standard linear algebra,
\[\tfrac{1}{2}\sum_{k=1}^{d-1}\theta_k^2 \leq\sum_{k=1}^{d-1}\theta_k^2+\sum_{1\leq j<k \leq d-1}\theta_j\theta_k \leq \tfrac{d+1}{2}\sum_{k=1}^{d-1}\theta_k^2.\]
Using this, we can conclude that 
\[ \delta (\phi, \hat{\zeta})\approx  \sum_{k=1}^{d-1}\theta_k^2,\]
for $\hat{\zeta} \in \mathbb{T}^{d-1}$ in some small neighborhood $V_{\epsilon}(\vec{1})$. This can be used to show that for $x$ sufficiently large,
\[ \mu(\Omega_x \cap V_{\epsilon}(1)) \approx \mu\left ( \hat{\zeta} \in V_{\epsilon}(\vec{1})\setminus \{\vec{1}\}:  \sum_{k=1}^{d-1}\theta_k^2 \le \frac{1}{x} \right) \approx \left(\frac{1}{x}\right)^{\frac{d-1}{2}}.\]
Then 
\[ \int_1^{\infty} \mu(\Omega_x \cap V_{\epsilon}(\vec{1})) \ x^{\p-2} \ dx \approx \int_1^{\infty} x^{\p-2-(\frac{d-1}{2})} \ dx,  \]
which is finite if and only if $\p<\tfrac{1}{2}(d+1).$ \end{example}

\subsection{Unimodular Level Sets}

In \cite{BPS18, BPS19}, the integrability properties of two-variable RIFs $\phi$ were studied using the 
behavior of unimodular level sets of $\phi$ near singularities on $\mathbb{T}^2$. Whether in two or $d$ variables, there are two standard ways to define the unimodular level sets
associated to such a $\phi = \frac{\tilde{p}}{p}$:
\[ \mathcal{C}_{\lambda} := \text{Closure}\left( \{ \zeta \in \mathbb{T}^d: \phi(\zeta) = \lambda\}\right) \ \ \text{ and } \ \ \mathcal{L}_{\lambda} : = \{ \zeta \in \mathbb{T}^d : \tilde{p}(\zeta) = \lambda p(\zeta)\} = \mathcal{C}_{\lambda} \cup (\mathbb{T}^d \cap \mathcal{Z}_{p}),\]
where $\lambda \in \mathbb{T}$.
There is a trade-off between considering these two sets. The set $\mathcal{C}_{\lambda}$ seems more closely related to $\phi$ and so, is more useful when we are identifying properties of $\phi$. In contrast, $\mathcal{L}_{\lambda}$ is the zero set of the polynomial $\tilde{p} - \lambda p$ restricted to $\mathbb{T}^d$ and so, is fairly easy to study. In \cite{BPS19}, the authors observed that in the two-variable setting, \cite[Corollary $1.7$]{Pas17} implies that these two definitions coincide. In the $d$-variable setting, we require a more complicated argument, but the final result still holds:

\begin{theorem} \label{thm:CL} Let $\phi = \frac{\tilde{p}}{p}$ be a RIF on $\mathbb{D}^d$. Then $\mathcal{C}_{\lambda} = \mathcal{L}_{\lambda}$ for all $\lambda \in \mathbb{T}$. 
\end{theorem}

\begin{proof} To prove equality, we need only show that $\T^d \cap \mathcal{Z}_p \subseteq \mathcal{C}_{\lambda}.$
By way of contradiction, fix $\tau \in \mathcal{Z}_p \cap \mathbb{T}^d$ and assume that there is some $\lambda \in \mathbb{T}$ 
so $\phi$ omits $\lambda$ in a neighborhood $\mathcal{U} \subseteq \mathbb{T}^d$ of $\tau.$ 
Let $\Pi$ denote the upper half plane and let $\alpha_1, \dots, \alpha_d$ and $\beta$ be 
conformal maps satisfying $\alpha_i : \Pi \rightarrow \mathbb{D}$ with 
$\alpha_i(0)=\tau_i$ and $\beta: \mathbb{D} \rightarrow \Pi$ with $\beta(\lambda)=\infty.$
Define $f: \Pi^d \rightarrow \Pi$ by $ f = \beta \circ \phi \circ (\alpha_1, \dots, \alpha_d)$. Let $\widehat{S} = \alpha^{-1} (\mathcal{Z}_p \cap \mathbb{T}^d)$. Then since $\dim(\mathcal{Z}_p \cap \mathbb{T}^d) \le d-2$, $\dim \widehat{S} \le d-2.$

Then in $\mathbb{R}^d$ near $0$, $f$ can only be singular on $\widehat{S}$ and $\alpha^{-1}(\mathcal{C}_{\lambda}).$ Since $\phi$ omits $\lambda$ near $\tau$, 
there is a  neighborhood $\mathcal{V} \subseteq \mathbb{R}^d$ of $0$ where $f$ is continuous on
 $\mathcal{V} \setminus \widehat{S}.$  We claim that the small size of $\widehat{S}$ will allow us to force $f$ to be analytic at $0$, which will yield the contradiction.
  To that end, choose a sequence $(x_n)\rightarrow 0$ so that
 \begin{itemize}
 \item[i.] $f$ is  continuous at each $x_n$; 
 \item[ii.] There is a fixed $\epsilon>0$ so that $B_{\epsilon}(x_n) \subseteq \mathcal{V}$ for all $n$. 
 \end{itemize}
 For each $n$, define the function $f_n(z):  = f(z+x_n)$. Then each $f_n: \Pi^d \rightarrow \Pi$ can be extended to be analytic in $\Pi^d \cup -\Pi^d$, is continuous in some neighborhood $B_n$ of $0$, and in $B_{\epsilon}(0)$, is only discontinuous at points $x$ where $x + x_n \in \widehat{S}$. We need to create a star-like wedge that omits those points.
 
Specifically, let $\ell_x$ denote the portion of the line between $0$ and $x$ that lies in $\overline{B_{\epsilon}(0)}$ and let $P(x)$ denote the two points in the intersection of that line with the boundary: $P(x): = \ell_x \cap \partial B_{\epsilon}(0)$. Define
\[ W_n = \left\{ x \in B_{\epsilon}(0): (\ell_x +x_n) \cap \widehat{S} =\emptyset \right\},\]
where $\ell_x +x_n$ denotes the shifted line segment.
We claim that $\mu(B_{\epsilon}(0) \setminus W_n)=0.$  To see this, observe that 
\[ 
\begin{aligned}
B_{\epsilon}(0) \setminus W_n &= \{ x \in B_{\epsilon}(0): (\ell_x +x_n) \cap \widehat{S} \ne \emptyset\} \\
& = \left\{ tx \in \mathbb{R}^d: x \in P\big( \big( \widehat{S} \cap B_{\epsilon}(x_n) \big) - x_n\big) \ \text{ and } -1 < t <1\right\}. 
 \end{aligned}
 \]
 Spherical coordinates and the fact that $\dim \widehat{S}  \le d -2$ implies that $\mu(B_{\epsilon}(0) \setminus W_n ) =0$ and so, 
\[ W: = \cap W_n \]
 has positive measure. Then $W$ is \emph{real wedge} in the sense of \cite{Pas19}. For each $n$, $f_n$ is continuous on $\Pi^d \cup W \cup B_n \cup -W\cup -\Pi^d$ and analytic on $\Pi^d \cup -\Pi^d$.  Then by Theorem 2.1 in \cite{Pas19} there is an open set $D \subseteq \mathbb{C}^d$ only depending on $W$ so that $f_n$ analytically continues to $D$. Letting $x_n$ get sufficiently close to $0$, this implies that $f$ analytically continues to $0$, a contradiction.

Thus $\mathbb{T}^d \cap \mathcal{Z}_{p} \subseteq \mathcal{C}_{\lambda}$, and the result follows. 
\end{proof}

For complicated RIFs, the conclusions of Theorem \ref{thm:CL} are not at all obvious. 

\begin{example} Consider degree $(2,1,1)$ RIF defined as follows:

\begin{equation}
\phi(z)=\frac{\tilde{p}(z)}{p(z)}=\frac{1-z_1-z_1z_2+z_1^2z_2-2z_3-z_1z_3-z_1^2z_3+z_2z_3-z_1z_2z_3+4z_1^2z_2z_3}{4-z_1+z_1^2-z_2-z_1z_2-2z_1^2z_2+z_3-z_1z_3-z_1z_2z_3+z_1^2z_2z_3}.\label{ex:badcurveRIF1}
 \end{equation}
This RIF is constructed and studied in detail in Example \ref{ex:curve2}. For now, it serves as a nice illustration of Theorem \ref{thm:CL}.
In particular, if we think about $\mathcal{Z}_{p}\cap \T^3$ as living in $[-\pi, \pi]^3$, we can represent it as
\[\mathcal{Z}_{p}\cap \T^3=\{(0,0, u)\colon u \in [-\pi, \pi]\}\cup \{(s, \arg m(e^{is}) ,\pi)\colon s \in [-\pi, \pi]\},\]
where $m(z)=\tfrac{3+z^2}{1+3z^2}$
is unimodular on $\mathbb{T}$. Thus, Theorem \ref{thm:CL} implies that every $\mathcal{C}_{\lambda}$ 
contains both curves of $\mathcal{Z}_{p}\cap \T^3$. The containment of the line $(0,0,u)$ actually forces all generic unimodular level sets $\mathcal{C}_{\lambda}$
to have a singularity at $(1,1)$ and so, unlike the unimodular level curves for two variable RIFs, these $\mathcal{C}_{\lambda}$ need not be smooth. This is illustrated in Figure \ref{fig:badcurveplots} below. 

\begin{figure}[h!]
 \includegraphics[width=0.5 \textwidth]{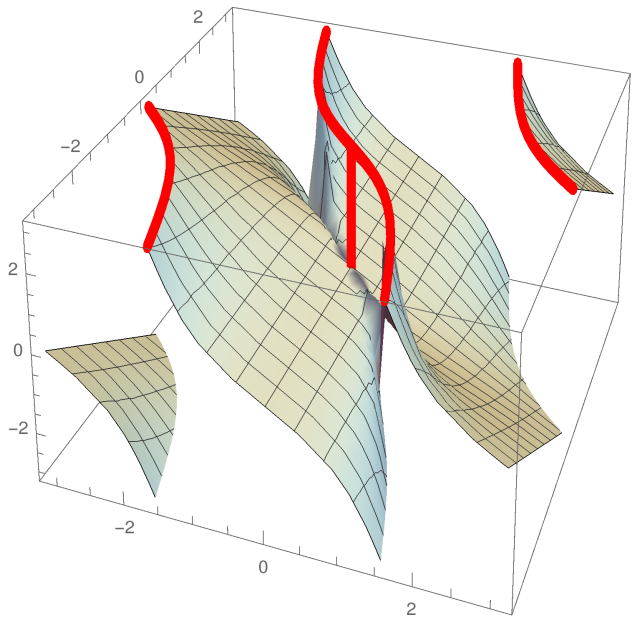}
  \caption{\textsl{The zero set for $\phi$ in \eqref{ex:badcurveRIF1} on $\mathbb{T}^3$ with a generic unimodular level set $\mathcal{C}_{\lambda}$}.}
  \label{fig:badcurveplots}
\end{figure}

\end{example}

\section{$3$-variable RIFs: Integrability} \label{sec:3var1}

To see both obstructions to a general theory and specific analytic results, we now restrict to three-variable irreducible RIFs $\phi$ with degree $(m,n,1)$. Such functions have better properties than general three-variable RIFs, see for instance \cite{Kne11PAMS},
 and thus provide a more tractable but still rich setting for investigating integrability, zero set, and unimodular level set questions. n the $(m,n,1)$ setting, we can write $\phi = \frac{\tilde{p}}{p}$ where 
\begin{equation} \label{eqn:mn1}  p(z) = p_1(z_1, z_2) + z_3 p_2(z_1, z_2), \ \ \  \tilde{p}(z) = z_3 \tilde{p}_1(z_1,z_2) + \tilde{p}_2(z_1,z_2), \end{equation}
and $p_1$, $p_2$ and $\tilde{p}_1, \tilde{p}_2$ respectively share no common factors. Here, the $\tilde{}$ notation means that the reflection operation is always taken with respect to $(m,n,1)$ or $(m,n)$. 

\subsection{Integrability Setup} When $\tilde{p}_1 \ne 0$, one can parameterize $\mathcal{Z}_{\tilde{p}}$ as $z_3 = \psi^0(z_1,z_2) := -\frac{\tilde{p}_2}{\tilde{p}_1}(z_1,z_2)$. Since $p$ is nonvanishing on $\mathbb{D}^3$, we have $|p_1| \ge |p_2|$ on $\mathbb{D}^2.$ On $\mathbb{T}^2$, this translates to 
\begin{equation} \label{eqn:p}  |\tilde{p}_1| = |p_1 | \ge |p_2| = |\tilde{p}_2|.\end{equation}
Thus, if $\tilde{p}_1$ vanishes at some $(\zeta_1, \zeta_2) \in \mathbb{T}^2$, then $\tilde{p}_2$ vanishes at $(\zeta_1, \zeta_2)$ as well and so, the vertical line $\{ (\zeta_1, \zeta_2) \} \times \mathbb{T}$ must be in $\mathbb{Z}_{\tilde{p}} \cap \mathbb{T}^3.$  Then, the set $\mathcal{Z}_{\tilde{p}_1} \cap \mathbb{T}^2$ must have measure $0$ and 
 the arguments in (and immediately proceeding) the proof of Theorem \ref{thm:GenInt} imply that
\begin{equation} \label{eqn:psi} \int_{\mathbb{T}^3} \Big | \tfrac{\partial \phi}{\partial z_3} (\zeta ) \Big|^\p |d\zeta | 
\approx \int_{\mathbb{T}^2} \left( 1 - |\psi^0(\zeta_1, \zeta_2) | ^2\right)^{1-\p} |d\zeta_1| |d\zeta_2| 
= \iint_{[-\pi,\pi]^2} \left | \rho_{\phi}(\theta_1, \theta_2) \right |^{1-\p} d\theta_1 d\theta_2, \end{equation}
where $\rho_{\phi}( \theta_1, \theta_2):=  1 - |\psi^0(e^{i\theta_1}, e^{i\theta_2})|^2$. This reduces $\tfrac{\partial \phi}{\partial z_3}$ 
to a more tractable problem in the case when $\rho_{\phi}$ is smooth. In particular, this puts us in the general setting of \cite{G06} and similar works. 
In the remainder of this section, we will (1) reduce to the setting where the $\rho_{\phi}$ are smooth, 
(2) provide the definitions and results from \cite{G06}, and (3) prove results about the power series expansions of the 
$\rho_{\phi}$'s in this setting and combine them with (2) to yield integration results about 
$\tfrac{\partial \phi}{\partial z_3}.$ There are several limitations to these methods, which are illustrated via the following examples:

\begin{itemize}
\item[i.] RIF $\phi$ with vertical lines $\{(\zeta_1, \zeta_2)\} \times \mathbb{T}$ in $\mathcal{Z}_p$ cannot be easily studied using this method. Sometimes this is an artifact of the proof, but other times, their associated $\rho_{\phi}$'s are actually not smooth. See Examples \ref{ex:vl1} and \ref{ex:vl2}.
\item[ii.] Some $ \tfrac{\partial \phi}{\partial z_3}$ have subtle integrability properties that cannot be captured using results from \cite{G06}. See Examples \ref{ex:lifted} and \ref{ex:curveiso}.
\end{itemize}

\subsection{Vertical Lines in $\mathcal{Z}_p$}

To force $\rho_{\phi}$ to be smooth, we will generally make the following assumption: $\mathcal{Z}_p \cap \mathbb{T}^3$ contains no vertical lines $\{(\zeta_1, \zeta_2)\} \times \mathbb{T}$.  This assumption guarantees that $\tilde{p}_1$ is nonvanishing on $\mathbb{T}^3$, which in turn implies that $\rho_{\phi}$ is smooth. Furthermore, the following is immediate.

\begin{lemma} \label{lem:analytic} Let $\zeta \in \mathbb{T}^3 \cap \mathcal{Z}_{\tilde{p}}$. If the vertical line $\{(\zeta_1, \zeta_2)\} \times \mathbb{T} \not \in \mathcal{Z}_{\tilde{p}}$, then there is a neighborhood 
 $U\subseteq \mathbb{C}^3$ of $\zeta$ such that 
\[\mathcal{Z}_{\tilde{p}}\cap U=\{(z_1,z_2, \psi^0(z_1,z_2)) :(z_1,z_2) \in V\},\] 
for some open $V\subseteq \C^2$ and $\psi^0:=-\frac{\tilde{p}_2}{\tilde{p}_1}$ analytic  near $(\zeta_1, \zeta_2)$. 
\end{lemma}

However, while the assumption that $\mathcal{Z}_p \cap \mathbb{T}^3$ contains no vertical lines is sufficient to conclude that $\rho_{\phi}$ is smooth, it is not always necessary.

\begin{example} \label{ex:vl1} Consider the RIF $\phi = \frac{\tilde{p}}{p}$,  where 
\[ p(z) = (2-z_1-z_2) + z_3 \tfrac{1}{2}(2z_1 z_2 - z_1 -z_2).\]
Then $\mathcal{Z}_p \cap \mathbb{T}^3$ contains the vertical line $\{(1,1)\} \times \mathbb{T}$. However, $| \psi^0(\zeta_1, \zeta_2)| = \tfrac{1}{2}$ for $(\zeta_1, \zeta_2) \in \mathbb{T}^2 \setminus \{(1, 1)\}$. From this, we can conclude that $\rho_{\phi} = \tfrac{3}{4}$ a.e. and
\[ \int_{\mathbb{T}^3} \Big | \tfrac{\partial \phi}{\partial z_3} (\zeta ) \Big|^\p |d\zeta |  \approx \int_{\mathbb{T}^2} \left( 1 - |\psi^0(\zeta_1, \zeta_2) |^2 \right)^{1-\p} |d\zeta_1| |d\zeta_2| =  \int_{\mathbb{T}^2} \left( \tfrac{1}{4} \right)^{1-\p} |d\zeta_1| |d\zeta_2| < \infty,\]
for all $\p$ satisfying $1 \le \p < \infty.$  Furthermore let 
\[ \Phi(z) = \frac{2z_1z_2-z_1-z_2}{2-z_1-z_2}, \quad \text{ so } \quad \phi(z) = \frac{ z_3 \Phi(z) + \frac{1}{2}}{1 + \frac{1}{2}z_3 \Phi(z)}\]
on $\mathbb{D}^3$. From this, one can easily deduce that $| \frac{\partial \phi}{\partial z_1}(\zeta_1, \zeta_2, \zeta_3) | \approx | \frac{\partial \Phi}{\partial z_1}(\zeta_1, \zeta_2)|$ on  $\T^3 \setminus \left(\{(1,1)\} \times \T \right)$. In \cite{BPS18}, the authors showed that $\frac{\partial \Phi}{\partial z_1} \in L^\p(\mathbb{T}^2)$ if and only if $\p < \frac{3}{2}$. Thus, we can immediately conclude that $\frac{\partial \phi}{\partial z_1} \in L^\p(\mathbb{T}^3)$ if and only if $\p < \frac{3}{2}.$ The same conclusion holds for $ \frac{\partial \phi}{\partial z_2}$. This demonstrates that, unlike in the two-variable setting, the partial derivatives of singular, irreducible $3$-variable RIFs need not have the same critical integrability indices.
\end{example}

Now let us consider an example of a RIF $\phi$ for which the function $\rho_{\phi}$ is actually
discontinuous on $[-\pi, \pi]^2.$

\begin{example} \label{ex:vl2} To construct $\phi$, we first define the function $f(w) = \langle (A -w_1Y_1 - w_2Y_2-w_3Y_3)^{-1} v, v\rangle,$ where 
\[A=\left(\begin{array}{ccc}0 & 1 &0\\1 & 0 & 1\\0 & 1 &0\end{array}\right), Y_1=\left(\begin{array}{ccc}1 & 0 & 0\\0 & 0 & 0\\0 & 0 &0 \end{array}\right),  Y_2=\left(\begin{array}{ccc}0 & 0 & 0\\0 & 1& 0\\0 & 0 &\frac{1}{2} \end{array}\right), v =\left( \begin{array}{c} 1 \\ 0 \\ 0  \end{array}\right) \]
and $Y_3=\mathbb{I}-Y_1-Y_2,$ where $\mathbb{I}$ is the $3\times 3$ identity matrix. 
By the representation theory developed in \cite{ATDY16}, $f$ is a Pick function, that is, $f$ is analytic in the poly-upper half-plane $\Pi^3=\{(w_1,w_2,w_3)\in \mathbb{C}^3\colon \mathrm{Im} \, w_j>0\}$ and maps $\Pi^3$ into $\Pi$. 
After conjugating $f$ with the M\"obius transformation 
\begin{equation}
m(z)=i\frac{1-z}{1+z}
\label{moebmap}
\end{equation}
mapping the disk to the upper half-plane, we obtain the RIF
\[\phi(z)=\frac{-1+z_1-2z_1z_2-z_2^2+3z_1z_2^2+z_3+z_1z_3-2z_2z_3-4z_1z_2z_3-z_2^2z_3+5z_1z_2^2z_3}{-5+z_1+4z_2+2z_1z_2-z_2^2-z_1z_2^2-3z_3+z_1z_3+2z_2z_3-z_2^2z_3+z_1z_2^2z_3}.\]
Direct substitution reveals that $p$ vanishes on $\{(1,1)\} \times \mathbb{T}.$  
Then setting $\tilde{p}=0$ and solving for $z_3$ yields:
\[z_3=\psi^0(z_1,z_2)=\frac{1-z_1+2z_1z_2+z_2^2-3z_1z_2^2}{1+z_1-2z_2-4z_1z_2-z_2^2+5z_1z_2^2}.\]
Both $\psi^0$ and $| \psi^0 |$ are discontinuous at $(1,1) \in \mathbb{T}^2$.  This can be seen as follows. By an elementary limiting argument 
\[ \lim_{\mathbb{T}\ni \zeta_1\to 1} \psi^{0}(\zeta_1,\zeta_1)=  \lim_{\mathbb{T}\ni \zeta_1\to 1} \frac{1-\zeta_1+3\zeta_1^2-3\zeta_1^3}{1-\zeta_1-5\zeta_1^2+5\zeta_1^3} = -1.\]
On the other hand, the set $\{(\zeta_1,\zeta_2)\in \mathbb{T}^2\colon \psi^0(\zeta_1,\zeta_2)=-\frac{1}{3}\}$ can be parametrized by
\[\zeta_1=\gamma(\zeta_2)=\frac{2-\zeta_2+\zeta_2^2}{1-\zeta_2+2\zeta_2^2}.\]
As $\gamma(1)=1$, this implies
\[\lim_{\zeta_2\to 1}\psi^0(\gamma(\zeta_2), \zeta_2)=-\frac{1}{3},\]
and so $\psi^0$ and $| \psi^0|$, and hence $\rho_{\phi}$, are not continuous at $(1,1)$. This illustrates why, to guarantee the smoothness of $\rho_{\phi}$, we will restrict to RIFs without vertical lines in their zero sets.
\end{example}

It is worth pointing out the following more general fact concerning vertical lines.

\begin{lemma} \label{lem:Flines} For every $\phi =\frac{\tilde{p}}{p}$ as in \eqref{eqn:mn1}, there are at most finitely many $(\zeta_1, \zeta_2) \in \mathbb{T}^2$ so that $\mathcal{Z}_p \cap \mathbb{T}^3$ contains the vertical line $\{(\zeta_1, \zeta_2)\} \times \mathbb{T}.$
\end{lemma}

\begin{proof} Observe that $\mathcal{Z}_p$ contains the vertical line $\{ (\zeta_1, \zeta_2)\} \times \mathbb{T}$ if and only if $p_1(\zeta_1,\zeta_2) = 0=p_2(\zeta_1, \zeta_2).$ Since $\phi$ is irreducible, $p_1$ and $p_2$ share no common factors. Thus, this can happen for at most finitely many $(\zeta_1, \zeta_2).$ 
\end{proof}

\subsection{Newton Polygons and Integrability} \label{subsec:GT}
In what follows, we require some well known definitions and Greenblatt's integrability results from \cite{G06}. For our purposes, it suffices to assume that $f$ vanishes at $(0,0)$, is real-analytic in an $\mathbb{R}^2$ neighborhood of $(0,0)$, is not identically $0$ near $(0,0)$, and has the Taylor series expansion
\[ f(x,y) = \sum_{k, \ell=0}^{\infty}  a_{k, \ell} x^k y^\ell \]
near $(0,0).$
For each pair $(k, \ell),$ let $Q_{k,\ell} = \{ (s,t) \in \mathbb{R}^2: k \le s, \ell \le t\}.$  Then one can compute the following:
\begin{definition} The \emph{Newton polygon} of $f$, denoted $N(f)$, is the convex hull of all of the $Q_{k, \ell}$ for which $a_{k,\ell} \ne 0.$ Typically,  $N(f)$ has boundary consisting of an infinite horizontal ray, an infinite vertical ray, and a finite number of negatively-sloped line segments \cite{G06}. The \emph{Newton distance} of $f$, denoted $\Delta(f)$, is the smallest $\delta>0$ for which $(\delta, \delta) \in N(f)$. Geometrically, it is the intersection of the line $y=x$ with boundary of $N(f).$  

If $y=x$ intersects the boundary of $N(f)$ in the interior of a finite line segment, say with slope $-\tfrac{1}{\tilde{m}}$, define
\[ e_f = \inf_{a_{k,\ell} \ne 0} (k + \ell \tilde{m}),   \quad g_f(c) = \sum_{k +\ell \tilde{m}=e_f} a_{k, \ell} c^{\ell}, \quad  \bar{g}_f(c) = \sum_{k +\ell \tilde{m}=e_f} a_{k, \ell} (-1)^k c^{\ell},\]
and let $Z(f)$ denote the maximum order of a zero of $g_f$ or $\bar{g}_f$.
\end{definition}

Before proceeding, let us illustrate these objects with a brief example.

\begin{example} \label{Ex:GT1} Assume $f$ is analytic in a neighborhood of $(0,0)$ with Taylor series expansion
\begin{equation} \label{eqn:NP}  f(x,y) = a_{2,0} x^2 + a_{1,1} xy + a_{0,2}y^2 + \text{ higher order terms},\end{equation}
and $a_{2,0}, a_{0,2} \ne 0.$ 
\begin{figure}[h!]
 \includegraphics[width=0.34 \textwidth]{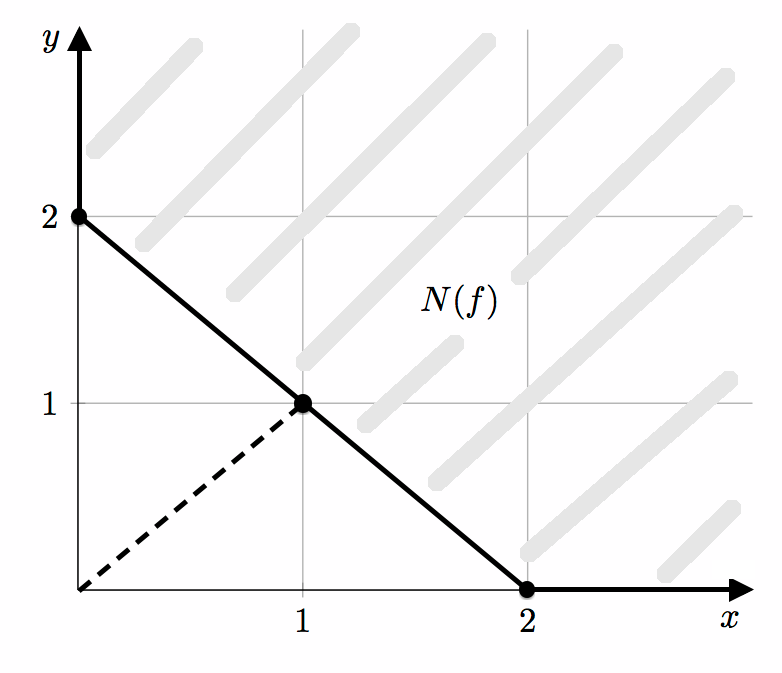}
  \caption{\textsl{The Newton polygon for $f$ in \eqref{eqn:NP}}.}
  \label{fig:NP}
\end{figure}
From Figure \ref{fig:NP}, one can see that  $y=x$ intersects the boundary piece $y = 2-x$ at the point $(1,1)$. Thus, $\Delta(f) =1$, $\tilde{m} = 1$, and $e_f = 2.$ This gives us the polynomials
\[ g_f(c) =  a_{2,0} + a_{1,1} c + a_{0,2}c^2 \ \ \text{ and } \ \  \bar{g}_f(c) =  a_{2,0} - a_{1,1} c + a_{0,2}c^2. \]
\end{example}

In \cite{G06}, Greenblatt characterized when the Newton distance $\Delta(f)$ can used to obtain the integrability index of $f$. Here is a summary of the main results, simplified to our current situation.

\begin{theorem}[Greenblatt, \cite{G06}] \label{thm:Green} Let $f$ be analytic near $(0,0)$ in $\R^2$, satisfy $f(0,0)=0$, and have a nonzero Taylor series expansion near $(0,0).$ Let $U$ be a sufficiently small neighborhood of $(0,0)$, which may depend on $f$, and define
\[ I^U_{\epsilon} = \int_U |f(x,y)|^{-\epsilon} dx \ dy.\]
Then the following hold:
\begin{itemize}
\item[a.]  If $y=x$ intersects the boundary of $N(f)$ in the interior of a finite line segment and $Z(f) \le \Delta(f)$, then  $ I^U_{\epsilon}  < \infty$ if and only if $\epsilon < \frac{1}{\Delta(f)}.$
\item[b.]   If $y=x$ intersects the boundary of $N(f)$ in the interior of a finite line segment and $Z(f) > \Delta(f)$, then there is an $\epsilon <\frac{1}{\Delta(f)}$ with $I^U_{\epsilon} = \infty$.
\item[c.] If $y=x$ intersects the boundary of $N(f)$ at a vertex, then  $ I^U_{\epsilon}  < \infty$ if and only if $\epsilon < \frac{1}{\Delta(f)}.$
\end{itemize} 

\end{theorem}

\begin{example} \label{Ex:GT2} Return to the setup in Example \ref{Ex:GT1} with
\[  f(x,y) = a_{2,0} x^2 + a_{1,1} xy + a_{0,2}y^2 + \text{ higher order terms}.\]
Then Theorem \ref{thm:Green} implies that 
\begin{itemize}
\item[a.] If $Z(f) = 1$, i.e. if $a_{1,1} \ne \pm 2 \sqrt{a_{2,0}a_{0,2}}$, then  $ I^U_{\epsilon}  < \infty$ if and only if $\epsilon < 1.$
\item[b.]  If $Z(f) = 2$, i.e. if $a_{1,1} = \pm 2 \sqrt{a_{2,0}a_{0,2}}$, then there is some $\epsilon <1$ with $ I^U_{\epsilon}  = \infty.$
\end{itemize}
In the RIF setting, we will see examples of both such cases.
\end{example}

\begin{remark}\label{rem:GT2}  There are also more complicated versions of Theorem \ref{thm:Green} for the situation where  $y=x$ intersects the boundary of $N(f)$ on an infinite line segment. As we have not yet observed this case in the setting of RIFs, we have omitted it here and refer the interested reader to \cite{G06}.
 However, it is worth noting that, amongst all of the cases, the best possible integrability index is $\epsilon = 1/\Delta(f)$. Thus, if an $f$ satisfies the conditions of Theorem \ref{thm:Green}, one will always have $I^{\epsilon}_U =\infty$ if $\epsilon > 1/\Delta(f)$. \end{remark}

\subsection{ $(m,n,1)$-Integration Results} \label{subsec:int}

As before, assume that $\phi$ is an irreducible degree $(m,n,1)$ RIF. Then \eqref{eqn:psi} indicates that studying the integrability of $\frac{\partial \phi}{\partial z_3}$ is equivalent to studying the integrability of $\rho_{\phi}$, where
\[\rho_{\phi}(\theta_1,\theta_2)= 1-|\psi^0(e^{i\theta_1}, e^{i\theta_2})|^2\]
and $\psi^0 = -\frac{\tilde{p}_2}{\tilde{p}_1}.$
To make the problem nontrivial, assume that $\phi$ has a singularity at some $\zeta \in \mathbb{T}^3$ and by changing variables, assume $\zeta = (1,1,1).$ To guarantee that $\rho_{\phi}$ is  real-analytic, assume that the vertical line $\{(1, 1)\}  \times \mathbb{T}$ is not in $\mathcal{Z}_{p}$. Then one can deduce a number of properties about the Taylor series expansion of $\rho_{\phi}$ near $(0,0).$

\begin{lemma}\label{lemma:Qformlemma}
Let 
\[\rho(\theta_1, \theta_2):=\rho_{\phi}(\theta_1,\theta_2)=\sum_{k=1}^{\infty}\sum_{\ell=1}^{\infty}c_{k,\ell}\theta_1^k\theta_2^\ell\]
be the Taylor series expansion of $\rho_{\phi}$ centered at the origin. Then:
\begin{enumerate}
\item We have $c_{0,0}=c_{1,0}=c_{0,1}=0$.
\item The quadratic form
\[Q(\theta_1,\theta_2)=Q_{\phi}(\theta_1,\theta_2)=c_{2,0}\theta_1^2+c_{1,1}\theta_1\theta_2+c_{0,2}\theta_2^2\] 
is positive semi-definite.

\item If $Q$ is strictly positive definite, then $(1,1,1)$ is an isolated singularity of $\phi$ on $\T^3.$
\item If $Q$ is identically $0$, then the order $3$ terms all vanish, namely:
\[ c_{3,0} = c_{2,1}=c_{1,2} = c_{0,3}=0.\] 
\end{enumerate}

\end{lemma}
\begin{proof}
To prove (1), observe that  $c_{0,0}=0$ follows directly from the assumption that $(1,1,1)\in \mathbb{T}^3$ is a singularity of $\phi$. Now write
\[ \rho(\theta_1,\theta_2)=c_{1,0}\theta_1+c_{0,1}\theta_2+R(\theta_1,\theta_2)\]
where $R(\theta_1,\theta_2)=\mathcal{O}(\|\theta\|^2)$ as $\theta \to (0,0)$.  Note that by \eqref{eqn:p}, $|\tilde{p}_2| \le |\tilde{p}_1|$ on $\T^2$ and so $\rho(\theta_1, \theta_2) \ge0$. If $c_{1,0}> 0$, say, we would have $c_{1,0}\theta_1<0$ for $\theta_1<0$ and $c_{1,0}\theta_1>0$ for $\theta_1>0$, which in turn would force a sign change in $\rho(\theta_1,0)$ for sufficiently small $\theta_1$. A similar argument shows that we cannot have $c_{1,0}<0$, meaning that $c_{1,0}=0$. Analogous reasoning involving $\rho(0,\theta_2)$ shows that $c_{0,1}=0$.

To establish ($2$), observe that
\[\rho(\theta_1,\theta_2)= Q(\theta_1,\theta_2)+ \mathcal{O}(\|\theta\|^3),\]
for $(\theta_1, \theta_2)$ near $(0,0)$. Thus, $Q(\theta_1,\theta_2)\geq 0$ near $(0,0)$, since $\rho(\theta_1,\theta_2)$ would otherwise attain negative values. Hence $Q$ is positive semi-definite. Conclusion ($3$) follows because  if $Q(\theta_1,\theta_2)$ is strictly positive definite, then $(0,0,0)$ is a strict local minimum of $\rho$. This in turn implies that $(1,1,1)$ is an isolated singularity of $\phi$ on $\T^3.$ 

To prove ($4$), note that for any $a\in \mathbb{R}$,
we must have
\[c_{3,0}\theta_1^3+c_{2,1}\theta_1^2(a\theta_1)+c_{1,2}\theta_1(a\theta_1)^2+c_{0,3}(a\theta_2)^3=\left(c_{3,0}+c_{2,1}a+c_{1,2}a^2+c_{0,3}a^3\right)\theta_1^3=0\]
for $\theta_1\in [-\pi,\pi]$.
Otherwise, the cubic form $c_{3,0}\theta_1^3+c_{2,1}\theta_1^2\theta_2+c_{1,2}\theta_1\theta_2^2+c_{0,3}\theta_2^3$ would have a sign chance at the origin
along the line $\{(\theta_1,a\theta_1)\} \subset [-\pi, \pi]^2$. This is turn would again imply that $\rho(\theta_1, \theta_2) <0$ near $(0,0)$, which cannot happen. Thus
\[
c_{3,0}+c_{2,1}a+c_{1,2}a^2+c_{0,3}a^3=0,
\]
which happens if $c_{3,0}=c_{2,1}=c_{1,2}=c_{0,3}=0$, or if $a\in\R$ is a root of the above cubic equation. But since the latter is possible for at most three distinct values of $a\in \mathbb{R}$, we deduce that the coefficients are equal to zero, as needed.
\end{proof}

We can combine Lemma \ref{lemma:Qformlemma} with Greenblatt's Theorem \ref{thm:Green} to conclude a number of results about the integrability of certain $\frac{\partial \phi}{\partial z_3}.$  Let us first point out how Theorem \ref{thm:Green} applies to the situation at hand.

\begin{remark} \label{rem:Green} Let $\phi = \frac{\tilde{p}}{p}$ be an irreducible RIF with $\deg \phi = (m,n,1)$ and assume $\mathcal{Z}_{\tilde{p}}\cap \T^3$ does not contain any vertical lines $\{(\zeta_1, \zeta_2)\} \times \mathbb{T}$. Further assume that $\phi$ has a singularity at $(1,1,1) \in \mathbb{T}^3$, let $\rho_{\phi}$ be as in  Lemma \ref{lemma:Qformlemma}, and let $U\subseteq \mathbb{R}^2$ be a sufficiently small neighborhood of $(0,0)$. Then 
there is a small neighborhood $V \subseteq \mathbb{T}^2$ of $(1,1)$ where an application of Lemma \ref{lem:fbp} implies that
\[  \iint_V \int_{\mathbb{T}} \Big | \tfrac{\partial \phi}{\partial z_3} (\zeta ) \Big|^\p |d\zeta | \approx \iint_{V} \left( 1-| \psi^0(\zeta_1, \zeta_2) |^2 \right )^{1 -\p} |d \zeta_1 | |d\zeta_2| = \iint_U \left | \rho_{\phi}(\theta_1, \theta_2)\right|^{-\epsilon} d\theta_1 d \theta_2,\]
for $\p>1$ and $\epsilon= \p-1 >0.$ As $\phi$ is singular at $(1,1,1)$, $\rho_{\phi}(0,0)=0.$ Moreover, $\rho_{\phi}$ is real analytic near $(0,0)$. If $\rho_{\phi}$ vanished identically near $(0,0)$, then there would be some open $\tilde{V} \subseteq \mathbb{T}^2$ with 
\[ \left \{ \left(\zeta_1, \zeta_2, \psi^0(\zeta_1, \zeta_2) \right): (\zeta_1, \zeta_2) \in \tilde{V} \right \} \subseteq \mathbb{T}^3 \cap \mathcal{Z}_{p}.\]
However, this is impossible because $\dim \left( \mathbb{T}^3 \cap \mathcal{Z}_p \right) \le 1.$ Thus, $\rho_{\phi}$ does not vanish identically and so, shrinking $U$ and $V$ if necessary, we can apply Theorem \ref{thm:Green} to obtain conditions on when  $\rho_{\phi}^{-1} \in L^{\epsilon}(U)$, or equivalently, when $\frac{\partial \phi}{\partial z_3} \in L^\p(V \times \mathbb{T})$ for $\p= 1+ \epsilon.$
\end{remark}

Here is a sampling of the results one can obtain by combining Lemma \ref{lemma:Qformlemma} and Theorem \ref{thm:Green}. To apply this result, one should first change variables to move the singularity of interest to $(1,1,1)$. 

\begin{theorem} \label{thm:GT2} Let $\phi = \frac{\tilde{p}}{p}$ be an irreducible RIF with $\deg \phi = (m,n,1)$ and assume $\mathcal{Z}_{\tilde{p}} \cap \T^3$ does not contain any vertical lines $\{(\zeta_1, \zeta_2)\} \times \mathbb{T}$. Further assume that $\phi$ has a singularity at $(1,1,1)$  and let $Q$ be as in Lemma \ref{lemma:Qformlemma}. Then
\begin{itemize}
\item[a.] If $c_{1,1} \ne \pm 2\sqrt{c_{0,2}c_{2,0}}$, then there is a neighborhood $V \subseteq \mathbb{T}^2$ of $(1, 1)$ such that $ \frac{\partial \phi}{\partial z_3 } \in L^\p(V \times \mathbb{T})$ if and only if $\p <2.$
\item[b.] If $c_{1,1} = \pm 2\sqrt{c_{0,2}c_{2,0}}$, then there is a neighborhood $V \subseteq \mathbb{T}^2$ of $(1, 1)$ and $\p <2$ such that $ \frac{\partial \phi}{\partial z_3 } \not \in L^\p(V\times \mathbb{T}).$
\end{itemize}
\end{theorem}

\begin{proof} By Lemma  \ref{lemma:Qformlemma}, the Taylor series expansion of $\rho_{\phi}$ centered at $(0,0)$ has $c_{0,0}, c_{0,1}, c_{1,0} =0$ and $Q$ positive semi-definite. 

Now consider (a). This additional restriction on the coefficients implies that $Q$ is strictly positive definite. Thus, both $c_{0,2}, c_{2,0}>0$ and so, $\rho_{\phi}$ is in the setting of Example \ref{Ex:GT2}. In particular, its Newton distance $\Delta(\rho_{\phi}) = 1$ and there is some neighborhood $U \subseteq \mathbb{R}^2$ of $(0,0)$ such that 
\[  \iint_U \left | \rho_{\phi}(\theta_1, \theta_2)\right|^{-\epsilon} d\theta_1 d \theta_2 < \infty \quad \text{if and only if} \quad \epsilon <1.\]
Then the discussion in Remark \ref{rem:Green} implies that there is a neighborhood $V \subseteq \mathbb{T}^2$ of $(1,1)$ such that $ \frac{\partial \phi}{\partial z_3} \in L^\p(V \times \mathbb{T})$ if and only if $\p = 1+ \epsilon <2.$

Now consider (b). We have two separate cases. First assume that $c_{0,2}, c_{2,0} \ne 0$. Then we are again in the setting of Example \ref{Ex:GT2}. The coefficient conditions imply that $\Delta(\rho_{\phi})=1$ and $Z(\rho_{\phi}) =2.$ Thus, there is some small $U \subseteq \mathbb{R}^2$ containing $(0,0)$ and $\epsilon <1$ so that 
\begin{equation} \label{eqn:W}  \iint_U \left | \rho_{\phi}(\theta_1, \theta_2)\right|^{-\epsilon} d\theta_1 d \theta_2 = \infty.\end{equation}
Thus, Remark \ref{rem:Green} again implies that there is some neighborhood $V\subseteq \T^2$ of $(1,1)$  so that $ \frac{\partial \phi}{\partial z_3} \not \in L^\p(V \times \mathbb{T})$ for $\p=1 +\epsilon<2$. Now assume without loss of generality that $c_{2,0} =0$. Then $c_{1,1}=0$ as well and a straightforward convexity argument shows that $(1,1)$ is not in the Newton polygon $N(\rho_{\phi})$. Thus $y=x$ must intersect $N(\rho_{\phi})$  beyond $(1,1),$ and the Newton distance $\Delta(\rho_{\phi}) >1$. Then the discussion in Remark \ref{rem:GT2} implies that there is some $\epsilon <1$ so that \eqref{eqn:W} holds, which again gives a neighborhood $V$ and $\p<2$ so that $ \frac{\partial \phi}{\partial z_3} \not \in L^\p(V \times \mathbb{T})$.
\end{proof}

The following is an immediate corollary of Theorem \ref{thm:GT2} and Lemma \ref{lemma:Qformlemma}. %characterization of when $Q$ is strictly positive-definite.

\begin{corollary} \label{cor:int} Let $\phi = \frac{\tilde{p}}{p}$ be a singular irreducible RIF with $\deg \phi = (m,n,1)$ such that $\mathcal{Z}_{\tilde{p}} \cap \mathbb{T}^3$ does not contain any vertical lines $\{(\zeta_1, \zeta_2)\} \times \mathbb{T}$. Then 
\begin{itemize}
\item[i.] If $\dim\left( \mathcal{Z}_p \cap \mathbb{T}^3\right)=0$, then $\frac{\partial \phi}{\partial z_3} \not \in L^\p(\mathbb{T}^3)$ for $\p \ge 2.$
\item[ii.] If $\dim\left(\mathcal{Z}_p \cap \mathbb{T}^3\right)=1$, then $\frac{\partial \phi}{\partial z_3} \not \in L^\p(\mathbb{T}^3)$ for $\p \ge \alpha,$ for some $\alpha <2.$
\end{itemize} \end{corollary} \begin{proof} For (i), observe that after changing variables to move the singularity to $(1,1,1)$,  $\phi$ must satisfy either $(a)$ or $(b)$ of Theorem \ref{thm:GT2}. Both imply that $\frac{\partial \phi}{\partial z_3} \not \in L^\p(\mathbb{T}^3)$ for $\p \ge 2.$ If $\dim\left(\mathcal{Z}_p \cap \mathbb{T}^3\right)=1$, then some singularity of $\phi$ is not isolated in $\T^3$. After changing variables, assume that this singularity is $(1,1,1).$ Then Lemma  \ref{lemma:Qformlemma}(3) implies that $Q$ cannot be strictly positive definite. This implies that $Q$ satisfies the condition in Theorem \ref{thm:GT2}(b), and so $\frac{\partial \phi}{\partial z_3} \not \in L^\p(\mathbb{T}^3)$ for $\p \ge \alpha,$ for some $\alpha <2.$
\end{proof}

\begin{remark} Most $z_3$-partial derivatives of singular degree $(m,n,1)$ RIFs should fail to be in $L^\p(\mathbb{T}^3)$ for $\p >2.$ However, Example \ref{ex:vl1} shows that there are singular RIFs whose zero sets contain vertical lines with better integrability. So, if one relaxed the condition that $\mathcal{Z}_{\tilde{p}}$ contain no vertical lines, then additional conditions on $\phi$ would be required. Moreover, we conjecture that the converse of Corollary \ref{cor:int}(i) is not true. In Example \ref{ex:curveiso2}, we provide a degree $(2,2,2)$ RIF with only isolated singularities whose $z_3$-partial derivative is in $ L^\p(\mathbb{T}^3)$ if and only if $\p <\frac{3}{2}.$ However, we have not found an irreducible degree $(m,n,1)$ RIF with both isolated singularities and worse derivative integrability than $\p <2.$
\end{remark} 

Let us consider the following example. Its derivative integrability was derived in a different way in Example \ref{ex:can}, but it also fits into this context.

\begin{example} \label{ex:iso1} Define the singular degree $(1,1,1)$ RIF $\phi$ by 
\begin{equation}
\phi(z)=\frac{3z_1z_2z_3-z_1z_2-z_1z_3-z_2z_3}{3-z_1-z_2-z_3}.
\label{ex:3dfave}
\end{equation}
As $\phi$ has a single isolated singularity at $(1,1,1) \in \mathbb{T}^3$, it suffices to study integrability near that point. 
Solving $\tilde{p}(z)=0$ for $z_3$ gives
\[z_3= \psi^0(z_1, z_2) = \frac{z_1z_2}{3z_1z_2-z_1-z_2}.\]
Then $\rho_{\phi}$ satisfies
\[ \begin{aligned} \rho_{\phi}(\theta_1, \theta_2) = 1 - \left | \psi^0\left( e^{i \theta_1}, e^{i\theta_2} \right ) \right|^2   
&= 1 - \frac{1}{11-6\cos \theta_1-6\cos\theta_2+2\cos(\theta_1-\theta_2)} \\
&= 2\theta_1^ 2+2\theta_1\theta_2+2\theta_2^2+\mathcal{O}(\|\theta\|^3),
\end{aligned}
\]
for  $(\theta_1,\theta_2)$ sufficiently close to $(0,0).$
Theorem \ref{thm:GT2} immediately implies that $\frac{\partial \phi}{\partial z_3} \in L^\p(\mathbb{T}^3)$ if and only if $\p <2.$

\end{example}

As demonstrated in Examples \ref{ex:vl1} and \ref{ex:vl2}, there are singular degree $(m,n,1)$ RIFs whose zero sets contain vertical lines of the form $\{(\zeta_1, \zeta_2) \} \times \mathbb{T}$. In the case of Example \ref{ex:vl2}, $\rho_{\phi}$ is discontinuous and so, the typical analysis assuming smoothness or real-analyticity \cite{PSS99, G06} does not apply. 

Similarly, there are RIFs where the analysis in Theorem \ref{thm:Green} does not provide the optimal integrability index. 

 \begin{example}\label{ex:lifted}Consider the RIF
\[\phi(z)=\frac{2z_1z_2z_3-z_1-z_2}{2-z_1z_3-z_2z_3}.\]
For this choice of $\phi$ we have
\[z_3=\psi^0(z_1,z_2)=\frac{1}{2z_1z_2}(z_1+z_2)\]
and then 
\[\rho_{\phi}(\theta_1,\theta_2)=1-\tfrac{1}{4}|1+e^{i(\theta_1-\theta_2)}|^2= \tfrac{1}{2} \cos(\theta_1-\theta_2) =\tfrac{1}{4}(\theta_1-\theta_2)^2+\mathcal{O}(\| \theta^4 \|) \]
has Newton distance $\Delta(\rho_{\phi})=1$. However, since this function is again obtained from the two-variable example $\Phi=(2z_1z_2-z_1-z_2)/(2-z_1-z_2)$, we have $\frac{\partial \phi}{\partial z_3}\in L^\p(\mathbb{T}^3)$ if and only if $\p<\frac{3}{2}$ by the results of \cite{BPS18} or by direct 
computation using the series expansion for $\cos(\theta_1-\theta_2)$. 

Cf. \cite[p.660-661]{G06} for a fuller discussion of this type of issue, and Example \ref{ex:curveiso} in Section \ref{sec:higherzoo} for another occurrence of this phenomenon in the RIF context.
\end{example}

We end this section by posing a problem.
\begin{question}
What possible configurations can arise in the Newton polygon associated with $\rho_{\phi}$, where $\phi$ is a three-variable RIF?
\end{question}

\section{3-variable RIFs: Boundary Values and Unimodular level sets} \label{sec:3var2}

 In this section, we continue our  study of three-variable irreducible RIFs $\phi$ with degree $(m,n,1)$. We will study both their non-tangential boundary values at points on $\mathcal{Z}_p \cap \mathbb{T}^3$ and the structure of their unimodular level sets $\mathcal{C}_{\lambda}$ on $\mathbb{T}^3$. 

\subsection{Boundary values}
Let us consider the non-tangential boundary values of $\phi$  on $\mathbb{T}^3$.  By non-tangential, we mean that $z \in \mathbb{D}^3$ approaches $\tau$ in a region where $\| z - \tau \| \le C (1- \|z\|)$, for some positive constant $C$. By \cite[Corollary 14.6]{Kne15}, every rational inner function $\phi$ has a non-tangential limit $\lambda \in \mathbb{T}$ at each $\tau \in \mathbb{T}^3$, which we denote by $\phi^*(\tau)$. We have the following characterization of non-tangential limit points:

\begin{theorem} \label{thm:lim} Let $(\tau_1, \tau_2) \in \mathbb{T}^2$. Then 
\begin{itemize}
\item[A.]   If $(\tau_1, \tau_2, \tau_3) \not \in \mathcal{Z}_p$ for all $\tau_3 \in \mathbb{T}$, then 
\[ \phi^*(\tau) = \phi(\tau) \qquad \text{ for all } \tau_3 \in \mathbb{T} \]
and for all $\lambda \in \mathbb{T}$, there is a unique $\tau_3\in \mathbb{T}$ with   $\phi^*(\tau)  = \lambda.$
\item[B.] If $(\tau_1, \tau_2, \zeta_3) \in \mathcal{Z}_p$ for exactly one $\zeta_3 \in \mathbb{T}$, then
\[ \phi^*(\tau) = \frac{\tilde{p}_2(\tau_1, \tau_2)}{p_1(\tau_1, \tau_2)} \in \mathbb{T} \qquad \text{ for all } \tau_3 \in \mathbb{T}. \]
\item[C.] If $\{(\tau_1, \tau_2)\}  \times \mathbb{T} \in \mathcal{Z}_p$, then $\mu : = \lim_{r\nearrow 1}  \frac{\tilde{p}_2(r \tau_1, r\tau_2)}{p_1(r\tau_1, r\tau_2)}$ exists and 
\begin{itemize}
\item[C1.] If $|\mu| =1$, for all $\tau_3$ (except possibly one), $\phi^*(\tau) = \mu.$
\item[C2.] If $|\mu| <1$, then for all $\lambda \in \mathbb{T}$, there is a unique $\tau_3 \in \mathbb{T}$ with $\phi^*(\tau)  = \lambda.$
\end{itemize}
\end{itemize}
\end{theorem}
\begin{remark} Note that Theorem \ref{thm:lim}B,C handle all points in $\mathcal{Z}_p \cap \mathbb{T}^d$. In particular, if $p$ is of the form \eqref{eqn:mn1}, then for each $\tau=(\tau_1, \tau_2,\tau_3) \in \mathcal{Z}_p \cap \mathbb{T}^3$ either $\tau_3 = -\tfrac{p_1}{p_2}(\tau_1,\tau_2)$ is the unique $\zeta_3 \in \mathbb{T}$ so that $(\tau_1, \tau_2,\zeta_3) \in \mathcal{Z}_p \cap \mathbb{T}^3$ or the entire vertical line $\{(\tau_1,\tau_2)\} \times \mathbb{T}$ is in $\mathcal{Z}_p \cap \mathbb{T}^d$. 
\end{remark} 

For RIFs with a finite singular set, we have  the following corollary:

\begin{corollary} \label{cor:1z} Let $\mathcal{Z}_{p} \cap \mathbb{T}^3$ be finite. Then for each 
$(\tau_1, \tau_2, \zeta_3) \in \mathcal{Z}_p \cap \mathbb{T}^3$, \[ \phi^*(\tau) = \frac{\tilde{p}_2(\tau_1, \tau_2)}{p_1(\tau_1, \tau_2)} \in \mathbb{T} \qquad \text{ for all } \tau_3 \in \mathbb{T}. \]
\end{corollary}

Here is the proof of Theorem \ref{thm:lim}:

\begin{proof}
For (A), the zero set assumption implies that $|p_2(\tau_1, \tau_2)| < |p_1(\tau_1, \tau_2)|.$ Then continuity implies that 
\[ \phi^*(\tau) = \phi(\tau) = \frac{ \tilde{p}_2(\tau_1, \tau_2)  + \tau_3 \tilde{p}_1(\tau_1, \tau_2)}{ p_1(\tau_1, \tau_2) + \tau_3 p_2(\tau_1,\tau_2)} = \frac{\tilde{p}_1}{p_1}(\tau_1, \tau_2) \left ( \frac{ \tau_3 + \frac{ \bar{p}_2}{\bar{p}_1}(\tau_1, \tau_2)} {1+ \tau_3\frac{p_2}{p_1}(\tau_1,\tau_2)}\right),\]
which is a Blaschke factor in $\tau_3$.  Thus, it is a one-to-one map from $\mathbb{T}$ to $\mathbb{T}$, which proves the claim.

For (B), the zero set assumption implies that $|p_2(\tau_1, \tau_2)| = |p_1(\tau_1, \tau_2)| \ne 0.$ If $\tau \not \in \mathcal{Z}_p$, continuity gives
\[  \phi^*(\tau) = \phi(\tau) = \frac{ \tilde{p}_2(\tau_1, \tau_2)  + \tau_3 \tilde{p}_1(\tau_1, \tau_2)}{ p_1(\tau_1, \tau_2) + \tau_3 p_2(\tau_1,\tau_2)}. \]
Because there is some $\zeta_3$ where both the numerator and denominator of $\phi$ vanish at ($\tau_1, \tau_2, \zeta_3)$, we can solve for $\zeta_3$ and conclude that 
\[  - \frac{p_1(\zeta_1, \zeta_2)}{p_2(\zeta_1, \zeta_2} = - \frac{\tilde{p}_2(\tau_1, \tau_2)}{\tilde{p}_1(\tau_1, \tau_2)}.\]
 Using that in the limit expression gives
\[  \phi^*(\tau) =  \frac{ \tilde{p}_2(\tau_1, \tau_2)}{ p_1(\tau_1, \tau_2)} \left( \frac{1 + \tau_3 \frac{ \tilde{p}_1}{\tilde{p}_2}(\tau_1, \tau_2)}{1 + \tau_3 \frac{p_2}{p_1}(\tau_2, \tau_2)}  \right) =  \frac{ \tilde{p}_2(\tau_1, \tau_2)}{ p_1(\tau_1, \tau_2)} .\]
The case where $\tau \in \mathcal{Z}_p$, equivalently when $\tau_3 = \zeta_3 = -\frac{p_1(\tau_1, \tau_2)}{p_2(\tau_1, \tau_2)},$ is proved in Lemma \ref{lem:lim1} below. 

For (C), observe that $\frac{\tilde{p}_1}{p_1}$, $\frac{p_2}{p_1}$ and $\frac{\tilde{p}_2}{p_1}$ are all holomorphic functions that are bounded by $1$ on $\mathbb{D}^2$. Then if $(r_n) \nearrow 1$, by passing to a subsequence, we can assume that 
\[ \frac{\tilde{p}_1}{p_1}(r_n \tau_1, r_n \tau_2) \rightarrow \alpha,  \ \  \frac{\tilde{p}_2}{p_1}(r_n \tau_1, r_n \tau_2) \rightarrow \beta, \ \ \text{ and }  \frac{p_2}{p_1}(r_n \tau_1, r_n \tau_2) \rightarrow \gamma,\]
for numbers $\alpha, \beta, \gamma$ with $\alpha \in \mathbb{T}$ and $\beta, \gamma \in \overline{\mathbb{D}}.$ Then for all $\tau_3 \in \mathbb{T}$ (except possibly if $\tau_3 =- \bar{\gamma}$), since $\phi^*(\tau)$ exists, we have
\[ \phi^*(\tau) =  \lim_{z \rightarrow \tau \ n.t.} \frac{ z_3\frac{\tilde{p}_1}{p_1}(z_1, z_2) + \frac{\tilde{p_2}}{p_1}(z_1,z_2)}{1 + z_3\frac{p_2}{p_1}(z_1,z_2)} = \frac{\tau_3 \alpha + \beta}{1+ \tau_3\gamma} \in \mathbb{T}.\]
This gives $\left| \tau_3 \alpha + \beta \right |^2 = \left| 1+ \tau_3\gamma \right|^2$ and expanding these out as trigonometric polynomials and comparing coefficients gives $\alpha \bar{\beta} = \gamma.$ Thus
\[  \phi^*(\tau) = \frac{\tau_3 \alpha + \beta}{1+ \tau_3\alpha \bar{\beta}}.\]
Now we have the two cases. If $\beta \in \mathbb{T}$, then for every $\tau_3 \in \mathbb{T}$, this limit is equal to $\beta$.  
Otherwise $\beta \in \mathbb{D}$, every $\tau_3 \ne -\bar{\gamma}$, and the limit formula is a Blaschke factor in $\tau_3;$ thus for each $\lambda \in \mathbb{T}$, there is a unique  $\tau_3 \in \mathbb{T}$ with $\phi^*(\tau)= \lambda.$

Finally, if $\lim_{r\nearrow 1}  \frac{\tilde{p}_2(r \tau_1, r\tau_2)}{p_1(r\tau_1, r\tau_2)}$ did not exist, then there would be some other sequence $(s_n) \nearrow 1$ so that 
\[ \frac{\tilde{p}_2}{p_1}(s_n \tau_1, s_n \tau_2) \rightarrow \nu \ne \beta.\]
Since $\left(\frac{\tilde{p}_1}{p_1}\right)^*(\tau_1, \tau_2)$ exists, it must happen that $\frac{\tilde{p}_1}{p_1}(s_n \tau_1, s_n \tau_2) \rightarrow \alpha$, the same value as before. These limits would imply different values for the $\phi^*(\tau)$, which is not possible. Thus,  $\mu:=\lim_{r\nearrow 1}  \frac{\tilde{p}_2(r \tau_1, r\tau_2)}{p_1(r\tau_1, r\tau_2)}$ is well defined and the  two cases occur when $|\mu|=1$ and $|\mu|<1.$
\end{proof}

\begin{remark} Both cases in Theorem \ref{thm:lim}C can occur. First consider $\phi = \frac{\tilde{p}}{p}$, where 
\[ p(z) = (2-z_1-z_2) +z_3\tfrac{1}{2}(2z_1 z_2 -z_1-z_2).\]
Then $\{(1,1)\} \times \T \subseteq \mathcal{Z}_p$ and $\lim_{r\nearrow 1}  \frac{\tilde{p}_2(r, r)}{p_1(r, r)} = \tfrac{1}{2}.$
Then for each $\tau_3 \in \T$, straightforward computation gives:
\[ 
\phi^*(1,1,\tau_3)  = \lim_{r \nearrow 1} \frac{ (2r^2-2r)r\tau_3 + \tfrac{1}{2}(2-2r)}{ (2-2r) + r \tau_3 \tfrac{1}{2}(2r^2-2r)} \\
 =  \lim_{r \nearrow 1}  \frac{ 2r^2 \tau_3 -1}{-2 +r^2 \tau_3} \\
 = - \frac{\tau_3 - \tfrac{1}{2}}{1 -\tfrac{\tau_3}{2}},
 \]
 a Blaschke factor in $\tau_3.$
 
Now consider $\psi = \frac{\tilde{p}}{p}$, where 
\[ p(z) = (2-z_1-z_2) +z_3\tfrac{1}{2}(1+z_1)(2z_1 z_2 -z_1-z_2).\]
Then $\{(1,1)\} \times \T \subseteq \mathcal{Z}_p$. Here, $\lim_{r\nearrow 1}  \frac{\tilde{p}_2(r, r)}{p_1(r, r)} = 1.$
Then for each $\tau_3 \in \T$ with $\tau_3 \ne 1$, a straightforward computation gives:
\[ 
\begin{aligned}
\psi^*(1,1, \tau_3)   = \lim_{r\nearrow 1} \frac{ r\tau_3(2r^3-2r^2) + \tfrac{1}{2} (1+r) (2-2r)}{(2-2r) + r\tau_3 \tfrac{1}{2} (r+1) (2r^2-2r)} 
= \lim_{r\nearrow 1} \frac{ 2r^3 \tau_3 - (1+r)}{-2 + r^2 \tau_3 (1+r)} 
 = \frac{2 \tau_3-2}{2\tau_3-2} =1.
\end{aligned}
\]
An application of L'H\^opital's rule implies that $\psi^*(1,1,1)=1$ as well.
\end{remark}

%For each $\lambda \in \mathbb{T}$,  define two-variable polynomial
%\[q_{\lambda}:= \lambda p_1-\tilde{p}_2.\]
%We can use this polynomial to determine when the nontangential boundary values of $\phi$ agree with the unimodular constant $\lambda.$
To finish the proof of Theorem \ref{thm:lim}, we need the following:

\begin{lemma} \label{lem:lim1} Assume $\tau =(\tau_1, \tau_2, \tau_3) \in \mathcal{Z}_p \cap \mathbb{T}^3$ and the vertical line $\{(\tau_1, \tau_2)\}  \times \mathbb{T}$ is not a component of $\mathcal{Z}_p$.  Then  $\phi^*(\tau) =\frac{\tilde{p}_2(\tau_1, \tau_2)}{p_1(\tau_1, \tau_2)}.$
\end{lemma}

\begin{proof} Since $\{(\tau_1, \tau_2)\}  \times \mathbb{T}$ is not a component of $\mathcal{Z}_p$, the polynomials $p_1$ and $p_2$ do not vanish at $(\tau_1, \tau_2)$ and
each  $\phi(\tau_1, \tau_2, r\tau_3)$  is well defined. By assumption, $\tau_3 = - \frac{p_1(\zeta_1, \zeta_2)}{p_2(\zeta_1, \zeta_2)} = - \frac{\tilde{p}_2(\tau_1, \tau_2)}{\tilde{p}_1(\tau_1, \tau_2)}.$ Then we can compute the facial limit
\[ 
\begin{aligned}
\lim_{r \nearrow 1} \phi(\tau_1, \tau_2, r \tau_3) &= \lim_{r \nearrow 1}  \frac{r \tau_3 \tilde{p}_1(\tau_1, \tau_2) + \tilde{p}_2(\tau_1,\tau_2)}{p_1(\tau_1, \tau_2) + r\tau_3 p_2(\tau_1, \tau_2)} \\
&= \lim_{r\nearrow 1} \frac{-r \tilde{p}_2(\tau_1,\tau_2) + \tilde{p}_2(\tau_1,\tau_2)}{p_1(\tau_1, \tau_2 ) - rp_1(\tau_1,\tau_2)}\\
 &= \frac{\tilde{p}_2(\tau_1, \tau_2)}{p_1(\tau_1, \tau_2)}.
\end{aligned}
\]
Now we need only show  $\phi^*(\tau)=\lim_{r \nearrow 1} \phi(\tau_1, \tau_2, r\tau_3).$ Assume $p$ vanishes to order $M$ at $\tau=(\tau_1, \tau_2, \tau_3)$. This means we can write 
\[ p(z) = P_M(z) + \sum_{j=M+1}^{m+n+1} P_j(z)\]
where each $P_j$ is homogeneous of degree $j$ in the terms $(z_1-\tau_1),(z_2-\tau_2),(z_3-\tau_3)$. As $p$ does not vanish identically when $z_1=\tau_1$ and $z_2=\tau_2$ and $\deg p = (m,n,1)$, this expansion must contain a term of the form $a(z_3-\tau_3)$.  As $M \ge 1$, this term must be part of $P_M$, so $M=1$. 
By Proposition 14.5 in \cite{Kne15}, there is a $\mu \in \mathbb{T}$ and homogeneous polynomials $Q_j$ in $(z_1-\tau_1),(z_2-\tau_2),(z_3-\tau_3)$ so that 
\[ \tilde{p}(z) = \mu P_1(z) + \sum_{j=2}^{m+n+1} Q_j(z).\]
Then Proposition 14.3 in  \cite{Kne15} implies that $\phi^*(\tau) = \mu.$ 
%Then since the non-tangential limit of $\phi$ is guaranteed to exist, we can restrict to a sequence $(z^k) \rightarrow \tau$ non-tangentially where $P_1(z^k) \approx (1-\| z\|)$. Then
%\[ \phi^*(\tau)  = \lim_{k \rightarrow \infty} \frac{  \mu P_1(z^k) + \sum_{j=2}^{m+n+1} Q_j(z^k)}{P_1(z^k) + \sum_{j=2}^{m+n+1} P_j(z^k)} = \mu.\]
We can also compute the facial limit
\[  \lim_{r \nearrow 1} \phi(\tau_1, \tau_2, r\tau_3) = \lim_{r \nearrow 1} \frac{\mu P_1(\tau_1, \tau_2, r\tau_3) + \sum_{j=2}^{m+n+1} Q_j(\tau_1, \tau_2, r\tau_3)}{ P_1(\tau_1, \tau_2, r\tau_3) + \sum_{j=2}^{m+n+1} P_j(\tau_1, \tau_2, r\tau_3)} = \lim_{r \nearrow 1}  \frac{\mu a(r \tau_3-\tau_3)}{a(r \tau_3-\tau_3) } =\mu.\]
Thus, $\mu = \frac{\tilde{p}_2(\tau_1, \tau_2)}{p_1(\tau_1, \tau_2)}$, as needed.
\end{proof}

\subsection{Unimodular Level Sets}

For each $\lambda \in \mathbb{T}$, define the polynomial
\begin{equation} \label{eqn:ql} q_{\lambda}:= \lambda p_1-\tilde{p}_2.\end{equation}
Then the unimodular level set $\mathcal{C}_{\lambda}$ can be described using this polynomial:

\begin{theorem} \label{lem:RIF} Fix $\lambda \in \mathbb{T}$, define $q_{\lambda}$ as in \eqref{eqn:ql}, and  let  $\Psi_{\lambda}:= \overline{\lambda} \frac{q_{\lambda}}{\tilde{q}_{\lambda}}.$ Then
\begin{itemize}
\item[(i)]  $\big( \mathcal{Z}_{q_{\lambda}} \cap \mathbb{T}^2 \big) \times \mathbb{T} \subseteq \mathcal{C}_{\lambda}.$ 
 \item[(ii)] $1/\Psi_{\lambda}$ is a two-variable RIF and 
 $\mathcal{C}_{\lambda} \setminus ( (\mathcal{Z}_{q_{\lambda}}\cap \mathbb{T}^2) \times \mathbb{T})$ is parameterized by 
\[ \tau_3= \Psi_{\lambda}(\tau_1, \tau_2) = \overline{\lambda} \frac{q_{\lambda}(\tau_1, \tau_2)}{\tilde{q}_{\lambda}(\tau_1, \tau_2)} \text{ for } (\tau_1, \tau_2)\in \mathbb{T}^2.\]
 \end{itemize}
\end{theorem}

\begin{proof}   

To show (i), use Theorem \ref{thm:CL} and observe that rewriting $\tilde{p} = \lambda p$ gives
\begin{equation} \label{eqn:ql2} z_3\left( \tilde{p}_1-\lambda p_2\right) = \lambda p_1-\tilde{p}_2 \ \ \text{ or equivalently, } \ \  z_3\lambda \tilde{q}_{\lambda} = q_{\lambda}.\end{equation}
If $(\tau_1, \tau_2) \in \mathcal{Z}_{q_{\lambda}} \cap \mathbb{T}^2$, then both sides vanish regardless of the value of $z_3.$ Thus, $\left( \mathcal{Z}_{q_{\lambda}} \cap \mathbb{T}^2 \right) \times \mathbb{T} \subseteq \mathcal{C}_{\lambda}.$

To show that $1/\Psi_{\lambda}$ is a RIF, we need only show that $q_{\lambda}$ is non-vanishing on $\mathbb{D}^2$. To see this, note that as $\phi$ is a RIF, $p_1$ is nonvanishing on $\mathbb{D}^2$ and on $\mathbb{T}^2$,
\[ \left | \tilde{p}_2(\tau_1, \tau_2) \right | = \left | p_2(\tau_1, \tau_2) \right| \le  \left | p_1(\tau_1, \tau_2) \right |.\]
This implies that $F:= \frac{\tilde{p}_2}{p_1}$ is holomorphic on $\mathbb{D}^2$ and by the maximum modulus principle on $\mathbb{D}^2$,  we either have $|F(z_1, z_2)| <1$ on $\mathbb{D}^2$ or $F(z_1, z_2) \equiv \alpha$ for some $\alpha \in \mathbb{T}$. This second option leads to a contradiction. Indeed if $F(z_1, z_2)= \alpha$, then $\tilde{p}_2 = \alpha p_1$ and $p_2 = \overline{\alpha} \tilde{p}_1$, which implies that 
\[ \phi(z_1,z_2,z_3) = \frac{ z_3 \tilde{p}_1(z_1,z_2) + \tilde{p}_2(z_1,z_2)}{p_1(z_1,z_2) + z_3 p_2(z_1,z_2)} = \frac{z_3\tilde{p}_1(z_1,z_2) + \alpha p_1(z_1,z_2)}{ p_1(z_1,z_2) + \overline{\alpha} z_3 \tilde{p}_1(z_1,z_2)} = \alpha,\]
a contradiction since $\phi$ is non-constant. This means that $|\tilde{p}_2(z_1,z_2)| < |p_1(z_1,z_2)|$ on $\mathbb{D}^2$ and $q_{\lambda}$ is nonvanishing on $\mathbb{D}^2$. 

To complete $(ii)$, observe that for every point $\tau \in \mathcal{C}_{\lambda} \setminus (\mathcal{Z}_{q_{\lambda}} \times \mathbb{T}),$ we can divide in \eqref{eqn:ql2} to obtain
\[ \tau_3 = \overline{\lambda} \frac{q_{\lambda}(\tau_1, \tau_2)}{\tilde{q}_{\lambda}(\tau_1, \tau_2)} = \Psi_{\lambda}(\tau_1, \tau_2),\]
as needed.
\end{proof}   

There is an interesting relationship between the zero set of $q_{\lambda}$ and non-tangential limits of $\phi$ at points in $\mathcal{Z}_p \cap \mathbb{T}^3$. This is encoded in the following:

\begin{lemma}  \label{lem:lim} Let $(\tau_1, \tau_2) \in \mathbb{T}^2$. Then 
\begin{itemize}
\item[A.] If $(\tau_1, \tau_2, \tau_3) \in \mathcal{Z}_p$ for exactly one $\tau_3 \in \mathbb{T}$, then
\[  \phi^*(\tau) = \lambda \text{ if and only if } q_\lambda(\tau_1, \tau_2) =0. \]
\item[B.] If $\{(\tau_1, \tau_2)\} \times \mathbb{T} \subseteq \mathcal{Z}_p$, then every $q_\lambda(\tau_1, \tau_2) =0$.
\end{itemize}
\end{lemma}

\begin{proof} Part (B) is trivial because the zero set assumption implies that 
\[ p_1(\tau_1,\tau_2) = \tilde{p}_2(\tau_1, \tau_2) =0.\]
This shows that, in the case of vertical lines in $\mathcal{Z}_p \cap \mathbb{T}^3$,  $q_{\lambda}$ does not govern the non-tangential limits.
Similarly, (A) follows immediately from Theorem \ref{thm:lim}B and the definition of $q_{\lambda}.$
\end{proof}

This allows us to study singularities of the $\mathcal{C}_{\lambda}$, at least when $\phi$ has only a finite number of singularities on $\mathbb{T}^3.$

\begin{theorem} \label{thm:z2} Assume $\mathcal{Z}_p \cap \mathbb{T}^3$ is a finite set. Fix $\lambda \in \mathbb{T}$. Then
\begin{itemize}
\item[A.] If there is a $\tau \in \mathcal{Z}_p \cap \mathbb{T}^3$ with $ \phi^*(\tau) = \lambda$, then $1/\Psi_{\lambda}$ has a singularity at $(\tau_1, \tau_2) \in \mathbb{T}^2$.
\item[B.] If there are no $\tau \in \mathcal{Z}_p \cap \mathbb{T}^3$ with  $\phi^*(\tau) = \lambda$, then $1/\Psi_{\lambda}$ is continuous on $\overline{\mathbb{D}^2}$.
\end{itemize}
\end{theorem}

\begin{proof} To prove (A), assume that there is a $\tau \in \mathcal{Z}_p \cap \mathbb{T}^3$ with $\phi^*(\tau) = \lambda$. Then Lemma \ref{lem:lim} implies that $q_{\lambda}(\tau_1, \tau_2) =0=\tilde{q}_{\lambda}(\tau_1, \tau_2)$. Since $\mathcal{Z}_p \cap \mathbb{T}^3$ is a finite set, $q_{\lambda}$ has at most finitely many zeros on $\mathbb{T}^2$. This implies that $q_{\lambda}$ is atoral, see \cite{AMS06}, and thus, $\q_{\lambda}$ and $\tilde{q}_{\lambda}$ share no common factors. It follows that $1/\Psi_{\lambda}$ must have a singularity at $(\tau_1, \tau_2)$.

To prove (B), assume that there is no $\tau \in \mathcal{Z}_p \cap \mathbb{T}^3$ with $\phi^*(\tau) = \lambda$. Then Lemma \ref{lem:lim} implies that $\mathcal{Z}_{q_{\lambda}} \cap \mathbb{T}^2= \emptyset$. Thus, $1/\Psi_{\lambda}$ does not have any singularities on $\mathbb{T}^2$ and $q_{\lambda}$ is atoral. Then \cite[Lemma 10.1]{Kne15} implies that $q_{\lambda}$ also does not vanish on $(\mathbb{D}\times \T) \cup (\T \times \D)$ and hence, $1/\Psi_{\lambda}$  is continuous on $\overline{\mathbb{D}^2}.$
\end{proof}

\begin{example}\label{ex:faverevisited}
 We will briefly use the canonical RIF $\phi$ defined in \eqref{ex:3dfave} to illustrate the results in this section. 
For this $\phi$, we have
\[ p_1(z) = 3-z_1-z_2 \ \ \text{ and } p_2(z) = -1.\]
Then $\mathcal{Z}_p \cap \mathbb{T}^3 = \{ (1,1,1)\}$ and one can easily compute that
\[ \phi^*(1,1,1) =\lim_{r \nearrow 1} \phi(r,r,r) =-1= \frac{\tilde{p}_2(1, 1)}{p_1(1, 1)}.\]
Moreover, Corollary \ref{cor:1z} also implies that for every $\tau_3 \in \mathbb{T}$, 
\[ \phi^*(1,1,\tau_3)= \frac{\tilde{p}_2(1, 1)}{p_1(1, 1)}=-1.\]
To study the unimodular level sets $\mathcal{C}_{\lambda}$ of $\phi$, observe that for each $\lambda \in \mathbb{T}$, $q_{\lambda}$ is defined by 
\[ q_{\lambda}(z_1,z_2) = \lambda(3-z_1-z_2)+z_1z_2\]
and $q_{-1}$ is the only $q_{\lambda}$ with a zero on $\mathbb{T}^2$. Then Theorem \ref{lem:RIF} implies that for  
$\lambda \ne -1$, the unimodular level surface $\mathcal{C}_{\lambda}$ is described by 
\[z_3 = \Psi_{\lambda}(z_1,z_2)= \bar{\lambda} \frac{q_{\lambda}(z_1, z_2)}{\tilde{q}_{\lambda}(z_1,z_2)} = \frac{3\lambda-\lambda z_1-\lambda z_2+z_1z_2}{\lambda-z_1-z_2+3z_1z_2}.\]
As implied by Theorem \ref{thm:z2}, we can see that each $1/\Psi_{\lambda}$ is a two-variable RIF continuous on  $\overline{\D^2}$.
In contrast, when $\lambda=-1$,  Theorem \ref{lem:RIF} implies that $\mathcal{C}_{-1}$ contains both the vertical line $\{(1,1)\} \times \mathbb{T}$ and the surface described by
\[z_3 = \Psi_{-1}(z_1,z_2) = \frac{-3+z_1+z_2+z_1z_2}{-1-z_1-z_2+3z_1z_2},\]
which has a singularity at $(1,1)$. A generic $\mathcal{C}_{\lambda}$ as well as the surface portion $\mathcal{C}_{-1}$ are
displayed in Figure \ref{faveplots}(a). Several unimodular level curves of the $\Psi_{\lambda}$ are included in  Figure \ref{faveplots}(b); their lack of common intersection points highlights the fact that these RIFs do not possess singularities.
\begin{figure}[h!]
    \subfigure[Unimodular level set $\mathcal{C}_{-1}$ (salmon) with a discontinuity and a generic smooth $\mathcal{C}_{\lambda}$.]
      {\includegraphics[width=0.44 \textwidth]{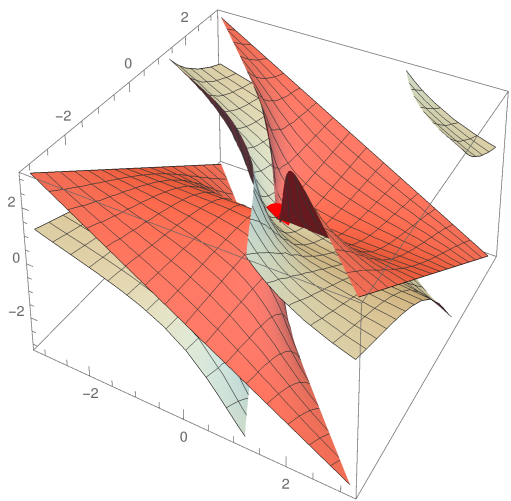}}
    \hfill
    \subfigure[Level curves for parametrizing $\Psi_{\lambda}$ for $\lambda=1$ (black) and $\lambda=\exp(i\pi/2)$ (green).]
      {\includegraphics[width=0.4 \textwidth]{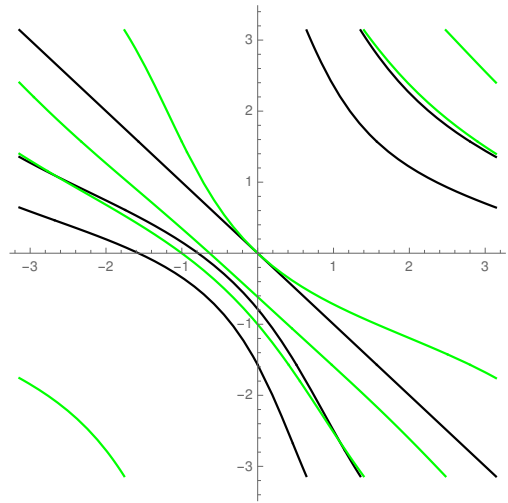}}
  \caption{\textsl{Level surfaces for $\phi$ in \eqref{ex:3dfave} and level curves of its parametrizing $\Psi_{\lambda}$}.}
  \label{faveplots}
\end{figure}

\end{example}

\section{Important RIF examples}\label{sec:higherzoo}

In this section, we illustrate a number of theorems and resolve several conjectures by constructing RIFs with specific properties. Here are the salient properties of each example. \\

\begin{itemize}
\item[i.] Example \ref{ex:curve} is a degree $(2,1,1)$ RIF $\phi = \frac{\tilde{p}}{p}$ whose singular set $\mathcal{Z}_p \cap \mathbb{T}^3$ is composed of three curves.  The function $\frac{\partial \phi}{\partial z_3} \in L^\p(\T^3)$ precisely when $\p < \frac{3}{2}$, its non-tangential boundary values illustrate Theorem \ref{thm:lim}B, and its unimodular level sets $\mathcal{C}_{\lambda}$ contain distinct vertical lines and have distinct singularities.  \\

\item[ii.] Example \ref{ex:curve2} is a degree $(2,1,1)$ RIF $\phi = \frac{\tilde{p}}{p}$ whose singular set $\mathcal{Z}_p \cap \mathbb{T}^3$ is composed of two curves, one of which is the vertical line $\{ (1,1)\} \times \mathbb{T}$.  The vertical line means that we cannot easily compute the integrability index of $\frac{\partial \phi}{\partial z_3}$; however, the function $\frac{\partial \phi}{\partial z_2} \in L^\p(\T^3)$ precisely when $\p <\frac{5}{4}.$  Its non-tangential boundary values illustrate Theorem \ref{thm:lim}C, and its generic unimodular level sets  $\mathcal{C}_{\lambda}$ all contain the same vertical line $\{(1,1)\} \times \mathbb{T}$ and have a common singularity at $(1,1).$ \\

\item[iii.] Example \ref{ex:curveiso} is a degree $(6,2,1)$ RIF $\phi = \frac{\tilde{p}}{p}$ whose singular set $\mathcal{Z}_p \cap \mathbb{T}^3$ is composed of an isolated point $(1,1,-1)$ and the curve $\{(-1, \zeta_2, -1): \zeta_2 \in \mathbb{T}\}.$ Near $(1,1,-1)$, $\frac{\partial \phi}{\partial z_3}$ is locally in $L^\p$ if and only if $\p <2.$ However, near the curve singularities, $\frac{\partial \phi}{\partial z_3}$ is locally in $L^\p$ if and only if $\p <\frac{3}{2}.$ Moreover, a direct computation rather than Greenblatt's Theorem \ref{thm:Green} must be used to deduce the $\p = \frac{3}{2}$ estimate. \\

\item[iv.] Example \ref{ex:curveiso2} is a degree $(2,2,2)$ RIF  $\phi = \frac{\tilde{p}}{p}$ whose singular set $\mathcal{Z}_p \cap \mathbb{T}^3$ is composed of the single point $(1,1,1)$.  A direct computation shows that  $\frac{\partial \phi}{\partial z_3} \in L^\p(\mathbb{T}^3)$ if and only if $\p < \frac{3}{2}$. This illustrates that RIFs with finite singular sets can exhibit the same derivative integrability as those with infinite singular sets. \\
\end{itemize}

\begin{example} \label{ex:curve}
To construct the first RIF, we combine representation formulas due to Agler, Tully-Doyle, and Young \cite{ATDY16} with ideas in \cite{Pas} as follows: consider the matrices
\[A=\left(\begin{array}{cccc}0 & 1 &0 & 1\\1 & 0 &1 & 0\\0 & 1 &0 & 1\\1 & 0 &1 &0\end{array}\right), \quad Y_1=\left(\begin{array}{cccc}1 & 0 & 0 & 0\\0 & 1 & 0 &0\\0 & 0 &0&0\\0 & 0 & 0 & 0 \end{array}\right), \quad Y_2=\left(\begin{array}{cccc}0 & 0 & 0 & 0\\0 & 0& 0 & 0\\0 & 0 &1 & 0\\ 0 & 0 & 0 &0 \end{array}\right),\]
and $Y_3=\mathbb{I}-Y_1-Y_2,$ where $\mathbb{I}$ is the $4 \times 4$ identity matrix. Define 
\begin{equation} 
f(w)=\langle (A-w_1Y_1-w_2Y_2-w_3Y_3)^{-1}v,v\rangle, \quad w \in \Pi^3,
\label{ATYFormula}
\end{equation}
where $v=(1,0,0,0)^T$. By \cite{ATDY16},
$f$ is a Pick function on the poly-upper half-plane $\Pi^3$. By conjugating with M\"obius maps of the form \eqref{moebmap} taking the polydisk to the poly-upper half-plane, one can obtain the degree $(2,1,1)$ RIF
\begin{equation}
\phi(z)=\frac{\tilde{p}(z)}{p(z)}=\frac{1-2z_1+z_1^2-z_2-2z_1z_2-z_1^2z_2+4z_1^2z_2z_3}{4-z_3-2z_1z_3-z_1^2z_3+z_2z_3-2z_1z_2z_3+z_1^2z_2z_3}.
\label{ex:nicecurveRIF}
\end{equation}
One can immediately see that as in \eqref{eqn:mn1},
\[ p_1(z_1,z_2) = 4 \ \ \text{ and } \ \ p_2(z_1, z_2) = -1-2z_1-z_1^2+z_2-2z_1z_2+z_1^2z_2.\]
Since $p_1$ does not vanish on $\mathbb{T}^3$, $\mathcal{Z}_{\tilde{p}}\cap\mathbb{T}^3$ does not contain a vertical line of the form $\{(\zeta_1, \zeta_2)\} \times \mathbb{T}$.
Moreover, Lemma \ref{lem:analytic} implies that if  $\psi^0= - \frac{\tilde{p}_2}{\tilde{p}_1}$, then $z_3: = \psi^0(z_1,z_2)$ parameterizes $\mathcal{Z}_{\tilde{p}}$ near $\mathbb{T}^3$.

Representing points on $\mathbb{T}^3$ using their arguments, one can check that
\begin{align*}
\mathcal{Z}_{p}\cap\mathbb{T}^3=&\{(0,t,0)\colon t\in [-\pi, \pi]\}\cup\{(s,0,-s) \colon s\in [-\pi, \pi]\}\\
 & \cup \{(\pi, t, -\pi-t) \colon \in t\in [-\pi, 0]\} \cup \{(\pi,t,\pi-t) \colon t\in [0, \pi]\},
 \end{align*}
so $\mathcal{Z}_{p}\cap\mathbb{T}^3$ is composed of three curves.

\begin{figure}[h!]
    \subfigure[$\mathcal{Z}_p \cap \mathbb{T}^3$ and a generic discontinuous $\mathcal{C}_{\lambda}$ ($\lambda=\exp(3i\pi/4)$) with vertical lines.]
      {\includegraphics[width=0.5 \textwidth]{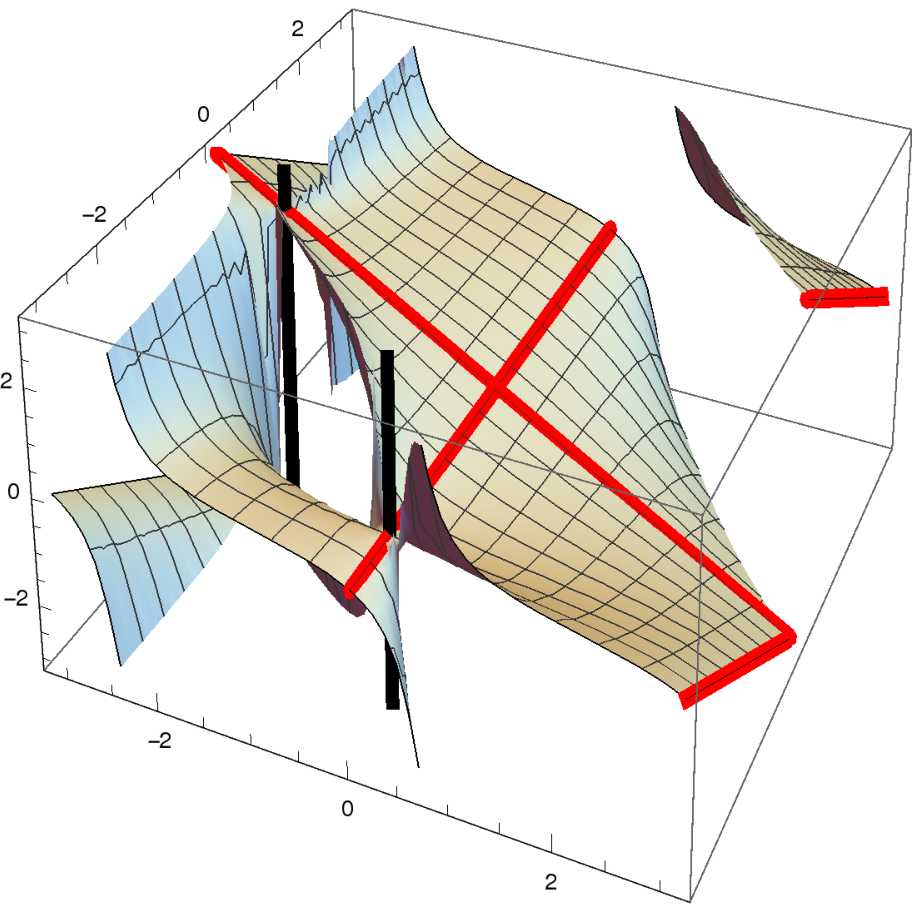}}
    \hspace{.5in}
    \subfigure[Level curves for parameterizing $\Psi_{\lambda}$ for $\lambda=\exp(3i\pi/4)$ (black) and $\lambda=\exp(i\pi/2)$ (green).]
      {\includegraphics[width=0.4 \textwidth]{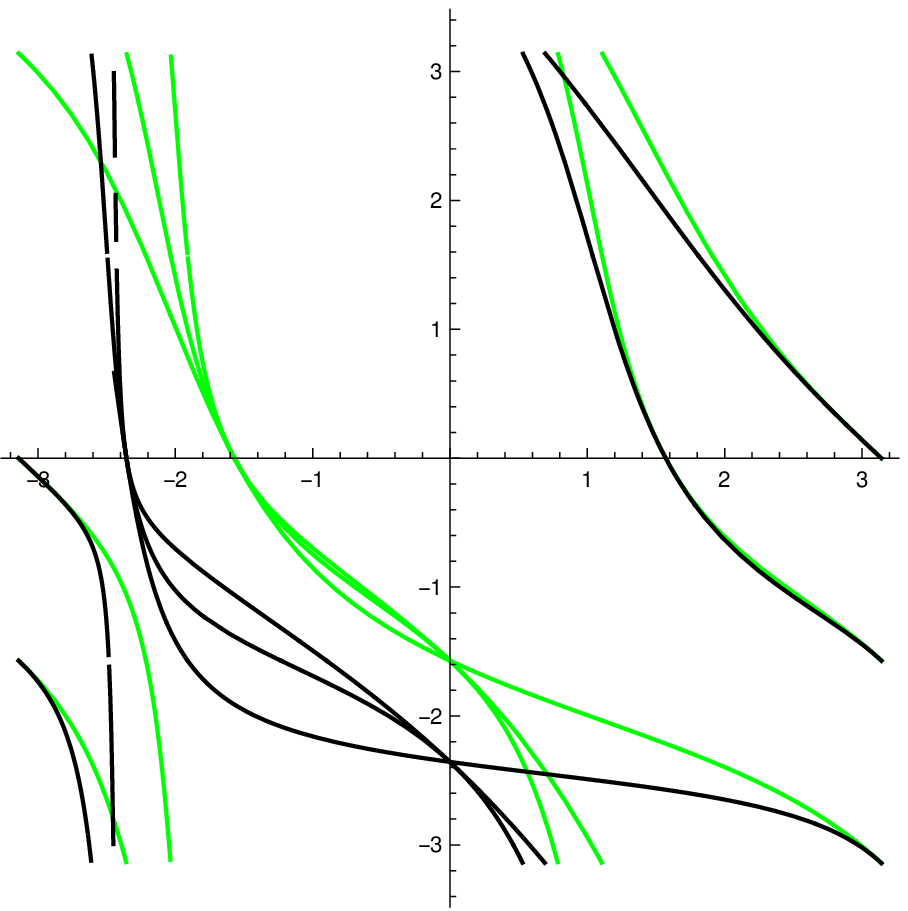}}
  \caption{\textsl{The zero set, a level set, and parameterizing functions for $\phi$ in
   \eqref{ex:nicecurveRIF}.}}
  \label{nicecurveplots}
\end{figure}
\end{example}

Let us consider $\phi$ in the context of Sections \ref{sec:3var1} and \ref{sec:3var2}. We first analyze its 
 derivative integrability, where an elementary computation reveals that 
\[\psi^0(z_1,z_2)=\frac{1}{4z_1z_2}\left(-1+2z_1-z_1^2+z_2+2z_1z_2+z_1^2z_2\right).\] 
Computing $\rho_{\phi}$ as in \eqref{eqn:psi} yields
\[ \rho_\phi(\theta_1,\theta_2)= 1- \left( \frac{1}{4}(3+\cos(2\theta_1)+\cos \theta_2-\cos (2\theta_1)\cos\theta_2) \right) = \theta_1^2\theta_2^2\left(\frac{1}{4}+\mathcal{O}(\|\theta\|)\right)\]
for $(\theta_1, \theta_2)$ near $(0,0)$. Computing successive partial derivatives of $\rho_{\phi}$ and evaluating at points of the form 
$(s,0)$ and $(t,0)$ one can show that $\rho_{\phi}$ has an expansion with bottom term of order $\theta_2^2$ and $\theta_1^2$, respectively.
Then as in \eqref{eqn:psi}, we have
\[\int_{\T^3}\left|\tfrac{\partial \phi}{\partial z_3}(\zeta)\right|^{\p}|d\zeta|\approx \iint_ {[-\pi, \pi]^2}\rho_{\phi}(\theta_1,\theta_2)^{1-\p}d\theta_1d\theta_2\]
and since $\rho_{\phi}(\theta_1,\theta_2)= (\theta_1\theta_2)^{2}\left(\frac{1}{4}+\mathcal{O}(\|\theta\|)\right)$ near the origin, we can conclude that
$\frac{\partial \phi}{\partial z_3} \in  L^{\p}(\T^3)$ if and only if $\p<\frac{3}{2}$. It is worth mentioning that while Theorem \ref{thm:Green} does apply here, the simplicity of $\rho_{\phi}$ makes its application unnecessary.

Now, as in Section \ref{sec:3var2}, consider the structure of $\phi$'s nontangential limits and unimodular level sets on $\mathbb{T}^3$. For each $\lambda \in \mathbb{T}\setminus\{-1\}$, there are exactly two points 
$(\tau^\lambda_{11},\tau^\lambda_{21})=(1, -\lambda)$ and $(\tau^\lambda_{12},\tau^\lambda_{22})=(-\lambda,1)$ on $\mathbb{T}^2$ so that for each $j$, the associated nontangential boundary value
 \[ \phi^*\Big( \tau^{\lambda}_{1j}, \tau^{\lambda}_{2j}, \psi^0\left (\tau^\lambda_{1j}, \tau^\lambda_{2j}\right) \Big) = \frac{\tilde{p}_2(\tau^\lambda_{1j}, \tau^\lambda_{2j})}{p_1(\tau^\lambda_{1j}, \tau^\lambda_{2j})}= \lambda.\]
 This can be seen by finding unimodular solutions to the equation $\tilde{p}_2(z_1,z_2)=4\lambda$
for each given $\lambda\in \mathbb{T}$.
 Thus, each $\lambda \in \mathbb{T}\setminus \{-1\}$ is the nontangential limit of $\phi$ associated to two points on $\mathcal{Z}_{\tilde{p}} \cap \mathbb{T}^3$ and if $\lambda=-1$ these points coincide. Furthermore, the structure of 
  $\mathcal{Z}_p$ puts us in the setting of Theorem \ref{thm:lim}B and allows us to compute all other nontangential values of $\phi$. In particular, as each $(\tau^\lambda_{1j}, \tau^{\lambda}_{2j}, \zeta_3) \in \mathcal{Z}_p$ for exactly one $\zeta_3:=  \psi^0(\tau^\lambda_{1j}, \tau^\lambda_{2j}) \in \mathbb{T}$, we obtain
\[ \phi^*(\tau^\lambda_{1j}, \tau^{\lambda}_{2j}, \tau_3) = \lambda \in \mathbb{T} \qquad \text{ for all } \tau_3 \in \mathbb{T}. \]
By Theorem \ref{lem:RIF}, each unimodular level set $\mathcal{C}_{\lambda}$ contains the two vertical lines $\{ (\tau^{\lambda}_{11}, \tau^{\lambda}_{21}) \}  \times \mathbb{T}$ and $\{ (\tau^{\lambda}_{12}, \tau^{\lambda}_{22}) \}  \times \mathbb{T}$ as well as the surface parameterized by
\[ z_3 = \Psi_{\lambda}(z_1,z_2)=\frac{1 - 4\lambda - 2z_1 + z_1^2 - z_2 - 2 z_1z_2 - z_1^2 z_2}{-\lambda - 2 \lambda z_1 - \lambda z_1^2 + \lambda z_2 - 
 2 \lambda z_1 z_2 - 4 z_1^2 z_2 + \lambda z_1^2 z_2}.\]
 A generic unimodular level set $\mathcal{C}_{\lambda}$ is displayed in Figure \ref{nicecurveplots}(a). 
It is worth noting that for each $\lambda \in  \T$, the parameterizing RIF $1/\Psi_{\lambda}$ has singularities at the points $(\tau^\lambda_{11}, \tau^\lambda_{21}), (\tau^\lambda_{12}, \tau^\lambda_{22}).$
Since the singularities are different for each $\lambda \in \mathbb{T}$,  the  $1/\Psi_{\lambda}$ have singularities at different points of $\T^2$. This is illustrated in Figure \ref{nicecurveplots}(b), where the $\Psi_{\lambda}$ have unimodular level curves clustering at different points.

\begin{example} \label{ex:curve2}
Applying the \eqref{ATYFormula} construction with  $v=(0,0,0,1)^T$ yields the degree $(2,1,1)$ RIF 
\begin{equation}
\phi(z)=\frac{\tilde{p}(z)}{p(z)}=\frac{1-z_1-z_1z_2+z_1^2z_2-2z_3-z_1z_3-z_1^2z_3+z_2z_3-z_1z_2z_3+4z_1^2z_2z_3}{4-z_1+z_1^2-z_2-z_1z_2-2z_1^2z_2+z_3-z_1z_3-z_1z_2z_3+z_1^2z_2z_3}.
\label{ex:badcurveRIF}
\end{equation}
It is immediate that as in \eqref{eqn:mn1},
\[ p_1(z_1,z_2) = 4-z_1 +z_1^2-z_2-z_1z_2-2z_1^2z_2\ \ \text{ and } \ \ p_2(z_1, z_2) = 1-z_1-z_1z_2+z_1^2z_2.\]
Here $\mathcal{Z}_{p}\cap \T^3$ has several features that sets it apart from the previous example. In particular, one can show that 
\[\mathcal{Z}_{p}\cap \T^3=\{(0,0, u)\colon u \in [-\pi, \pi]\}\cup \{(s, \arg m(e^{is}) ,\pi)\colon s \in [-\pi, \pi]\}\]
where $m(z)=\frac{3+z^2}{1+3z^2}$ has modulus $1$ for $z=e^{i\theta}$. 
%The first component can be extracted by first evaluating
%\[p(z_1,z_1,z_3)=4-2z_1-2z_1^2+z_3-z_1z_3-z_1^2z_3+z_1^3z_1\]
%and then letting $z_1\to 1$ non-tangentially, and the second component is found by examining
%\[p(z_1,z_2,1)=3+z_1^2-z_2-3z_1^2z_2=0\]
%and solving for $z_2$. 
Thus, $\mathcal{Z}_{p}\cap \T^3$ contains the vertical line $\{ (1,1)\} \times \mathbb{T}.$ The presence of this vertical line affects the integrability of $\phi$, the behavior of its non-tangential limits, and the structure of its unimodular level sets. First, since we cannot parametrize $z_3$ in terms of the variables $z_1$ and $z_2$, we cannot use the integrability results from Section \ref{sec:3var1} or sufficiently estimate the rate of growth of $\mu(\Omega_x)$ to apply Theorem \ref{thm:GenInt}.
Thus, we have been unable to compute the integrability index for $\frac{\partial \phi}{\partial z_3}.$  In contrast, one can adapt the arguments from Section \ref{sec:3var1} to compute the integrability index for $\frac{\partial \phi}{\partial z_2}.$ Because that computation is rather technical, we leave it to the end of this example.

Now we examine $\phi$ in the context of Section \ref{sec:3var2} and first consider the non-tangential boundary values of $\phi$. Because $\mathcal{Z}_{p}\cap \T^3$ contains the vertical line $\{ (1,1)\} \times \mathbb{T},$ if we consider $(1,1)$, we are in the setting of 
 Theorem \ref{thm:lim}C2. To apply it, one can check that
\[ \lim_{r \nearrow 1}  \frac{\tilde{p}_2(r,r)}{p_1(r,r)} = 0.\]
Then Theorem \ref{thm:lim} implies that for each $\lambda \in \mathbb{T}$, there is a unique $\tau_3 \in \mathbb{T}$ such that $\phi^*(1,1,\tau_3)=\lambda.$
A simple computation shows that this unique $\tau_3=-\lambda$. Considering any $(\tau_1, m(\tau_1))$ with $\tau_1 \ne 1 \in \mathbb{T}$ puts us in the setting of Theorem \ref{thm:lim}B. In this case, direct computation yields
\[ \phi^*(\tau_1, m(\tau_1), \tau_3) = \frac{ \tilde{p}_2(\tau_1, m(\tau_1))}{p_1(\tau_1, m(\tau_1))}=1 \quad \text{for all $\tau_3 \in \mathbb{T}$.}\]

These boundary value results allow us to deduce the structure of $\phi$'s unimodular level sets.
First, we can conclude that if $\lambda \ne 1$, $q_{\lambda}$ only vanishes at $(1,1)$ on $\mathbb{T}^2$. Then for each $\lambda \in \T$ with $\lambda \ne 1$, Theorem \ref{lem:RIF} implies that the unimodular level set $\mathcal{C}_\lambda$ contains exactly the vertical line $\{(1,1)\} \times \mathbb{T}$ and the surface
\[z_3=\Psi_{\lambda}(z_1,z_2)=\frac{1 - 4\lambda  - z_1 + \lambda z_1 - \lambda z_1^2 + \lambda z_2 - z_1 z_2 + \lambda z_1 z_2 + z_1^2 z_2 + 2 \lambda z_1^2 z_2}{2 + \lambda + z_1 - \lambda z_1 + z_1^2 - z_2 + z_1 z_2 - \lambda z_1 z_2 - 4 z_1^2 z_2 + \lambda z_1^2 z_2}.\]
A generic unimodular level set with $\mathcal{Z}_p \cap \mathbb{T}^3$ is shown in Figure \ref{badcurvelevellines}a.
Each  parametrizing RIF $1/\Psi_{\lambda}$ also has a single singularity at $(1,1)$. The unimodular level curves for several $\Psi_{\lambda}$  are shown in Figure \ref{badcurvelevellines}b: the fact that these all pass through $(0,0)$ reflects the common singularity at $(1,1)$.

\begin{figure}[h!]
    \subfigure[$\mathcal{Z}_p \cap \mathbb{T}^3$ and a generic discontinuous $\mathcal{C}_{\lambda}$.]
      {\includegraphics[width=0.5 \textwidth]{lightex2.eps}}
    \hspace{.5in}
    \subfigure[Level curves for parameterizing $\Psi_{\lambda}$. Different colors indicate different values of $\lambda$.]
      {\includegraphics[width=0.4 \textwidth]{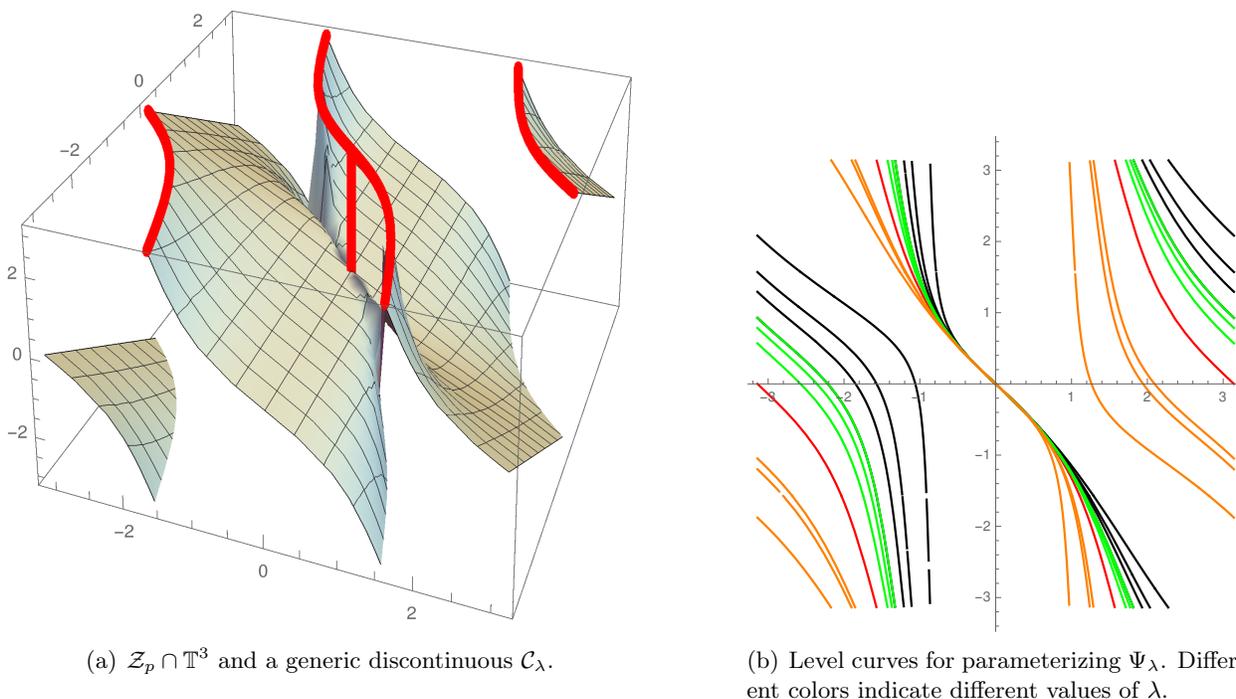}}
  \caption{\textsl{The zero set, a level set, and parameterizing functions for $\phi$ in
   \eqref{ex:badcurveRIF}.}}
  \label{badcurvelevellines}
\end{figure}

Finally consider $\lambda =1$. Then $q_1(\tau_1,\tau_2)$ vanishes whenever $\tau_2 = m(\tau_1)$. Thus Theorem \ref{lem:RIF} implies that the unimodular level set $\mathcal{C}_1$ contains every vertical line $\{(\tau_1, m(\tau_1))\} \times \mathbb{T}$ and the plane
\[ \tau_3 = \Psi_1(\tau_1, \tau_2) = -1.\]

Before ending this example, let us return to  $\frac{\partial \phi}{\partial z_2}$. It is worth investigating because its integrability is worse than has appeared in previous examples.
Solving $\tilde{p}(z)=0$ for $z_2$ gives
\[z_2= \psi^0(z_1,z_3):=\frac{-1+z_1+2z_3+z_1z_3+z_1^2z_3}{-z_1+z_1^2+z_3-z_1z_3+4z_1^2z_3}.\]
Evaluating for $(\tau_1, \tau_3)=(e^{i\theta_1}, e^{i\theta_3})$ and performing some computations, we find that
\[|\psi^0(\tau_1,\tau_3)|^2=\frac{4+2\cos \theta_1+2\cos\theta_2+2\cos(\theta_1-\theta_3)-\cos \theta_3-\cos(2\theta_1+\theta_3)}{10-6\cos\theta_1+4\cos(2\theta_1)-2\cos(\theta_1-\theta_3)+\cos(2\theta_1-\theta_3)+5\cos\theta_3-4\cos(\theta_1+\theta_3)}.\]
Next, we check that $|\psi^0(1,e^{i\theta_3})|^2=1$ and $|\psi^0(e^{i\theta_1},-1)|^2=1$ as expected. Let $d(\theta_1,\theta_3)$ and $n(\theta_1,\theta_3)$ denote the denominator and numerator in $|\psi^0(\theta_1,\theta_3)|^2$, respectively. First, we note that $d(\theta_1,\theta_3)$ is bounded below on $[-\pi, \pi]^2$. A computation using trigonometric identities shows that
\[d(\theta_1,\theta_3)-n(\theta_1,\theta_3)=32\sin^4\left(\frac{\theta_1}{2}\right)\cos^2\left(\frac{\theta_3}{2}\right).\]
Thus, we have
\[ |\psi^0(e^{i\theta_1},e^{i\theta_3})|^2=1-\frac{32}{d(\theta_1,\theta_3)}\sin^4\left(\frac{\theta_1}{2}\right)\cos^2\left(\frac{\theta_3}{2}\right),\]
meaning that the local integrability of $\frac{\partial \phi}{\partial z_2}$ is determined by the order of vanishing of the function
\[\varrho_{\phi}(\theta_1,\theta_3)=\sin^4\left(\frac{\theta_1}{2}\right)\cos^2\left(\frac{\theta_3}{2}\right)\]
along $\mathcal{Z}_p \cap \mathbb{T}^3$.  A straightforward expansion shows that
\[\varrho_{\phi}(\theta_1,\theta_3) = \theta_1^4\left(\frac{1}{16}-\frac{1}{64}\theta_3^2+\mathcal{O}(\theta_1^2)+\mathcal{O}(\theta_3^4)\right)\]
for $(\theta_1,\theta_3)$ close to $(0,0)$ and this order of vanishing gives the worst integrability present along the zero set. Thus, as in \eqref{eqn:psi},
\[\int_{\T^3}\left|\tfrac{\partial \phi}{\partial z_2}(\zeta)\right|^\p|d\zeta|\approx \int_{[-\epsilon,\epsilon]^2}\theta_1^{4-4p}\left(\tfrac{1}{16}-\tfrac{1}{64}\theta_3^2+\mathcal{O}(\theta_1^2)+\mathcal{O}(\theta_3^4)\right)^{1-\p}d\theta_1d\theta_3.\]
and $\frac{\partial \phi}{\partial z_2}\in L^{\p}(\T^3)$ if and only if $\p<\frac{5}{4}$.  Lastly, observe that near $(\pi,\pi)$, we have 
\[\varrho_{\phi}(\pi-\eta_1,\pi-\eta_2)\approx \eta_3^2(1-\frac{1}{2}\eta_1^2)\]
 to lowest order, meaning  that $\frac{\partial \phi}{\partial z_2}$ is locally in $L^\p$ near $(-1,-1)$ for $\p<\frac{3}{2}$.
\end{example}
\begin{question}
What is the critical integrability of $\frac{\partial \phi}{\partial z_3}$ for the above example? Is the presence of a joint vertical line for the level sets $\mathcal{C}_{\lambda}$ generically accompanied by worse integrability?
\end{question}
\begin{example} \label{ex:curveiso}
%The next example shows that a degree $(m,n,1)$ rational inner function can simultaneously possess isolated singularities and singularities along a curve. 
%Phrased differently, there are semi-stable polynomials in three variables that possess both isolated zeros and zero curves in $\T^3$.
Consider the RIF in \eqref{ex:nicecurveRIF}, compose it with the polynomial conformal mappings $z_1\mapsto -\frac{1}{4}(z_1^3+3z_1)$ and $z_2\mapsto -\frac{1}{3}(z_2^2-2z_2)$, and multiply through by $48$ in numerator and denominator to obtain integer coefficients. This yields a degree $(6,2,1)$ RIF $\phi$ where $\phi(z)=\frac{\tilde{p}(z)}{p(z)}$
with
\begin{multline}
p(z)=192 + 48 z_3 - 72 z_1 z_3 + 27 z_1^2 z_3 - 24 z_1^3 z_3 + 18 z_1^4 z_3 + 3 z_1^6 z_3 + 
 32 z_2 z_3 + 48 z_1 z_2 z_3 + 18 z_1^2 z_2 z_3 + 16 z_1^3 z_2 z_3 \\+ 12 z_1^4 z_2 z_3 + 
 2 z_1^6 z_2 z_3 + 16 z_2^2 z_3 + 24 z_1 z_2^2 z_3 + 9 z_1^2 z_2^2 z_3 + 8 z_1^3 z_2^2 z_3 + 
 6 z_1^4 z_2^2 z_3 + z_1^6 z_2^2 z_3.
\label{combofeature}
\end{multline}
This $p$ is irreducible since we have $p=192+P(z_1,z_2)z_3$ for a polynomial $P\in \C[z_1,z_2]$.
Now consider $\mathcal{Z}_p$. We have  $p(1,1,-1)=0$, and $(1,1,-1)$ is an isolated zero of $p$ in $\T^3$; moreover, $\phi^*(1,1,-1)=1$. 
Furthermore, one can check by direct substitution that $p(-1,z_2,-1)=0=\tilde{p}(-1,z_2,-1)$, and together with $(1,1,-1),$ this is all of $\mathcal{Z}_p\cap \mathbb{T}^3$. Thus, $\phi$ has a curve of singularities in $\T^3$ in addition to the isolated singularity at $(1,1,-1)$. 

We will not give an indepth analysis of the non-tangential boundary values and unimodular level sets of $\phi$. However, the zero set $\mathcal{Z}_p \cap \mathbb{T}^3$ and a generic unimodular level curve are given in Figure \ref{combocurveplots}(a). Level curves for $\Psi_{\lambda}$  associated with  $\lambda=-1$ (black) and $\lambda=1$ (green) are given in  Figure \ref{combocurveplots}(b). Note that the green level lines all pass through the origin, whereas only one black curve does, reflecting the fact that $\phi^*(1,1,-1)=1$. However, both black and green curves pick up singularities at $\theta_1=\pi$, corresponding to the curve component of $\mathcal{Z}_p$.

\begin{figure}[h!]
    \subfigure[$\mathcal{Z}_p \cap \mathbb{T}^3$ and a generic $\mathcal{C}_{\lambda}$.]
      {\includegraphics[width=0.5 \textwidth]{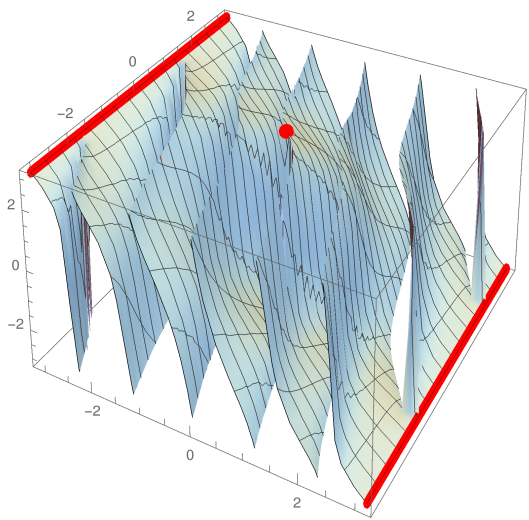}}
    \hspace{.5in}
    \subfigure[Level curves for parameterizing $\Psi_{\lambda}$ with $\lambda =1$ (green) and $\lambda=-1$ (black).]
      {\includegraphics[width=0.4 \textwidth]{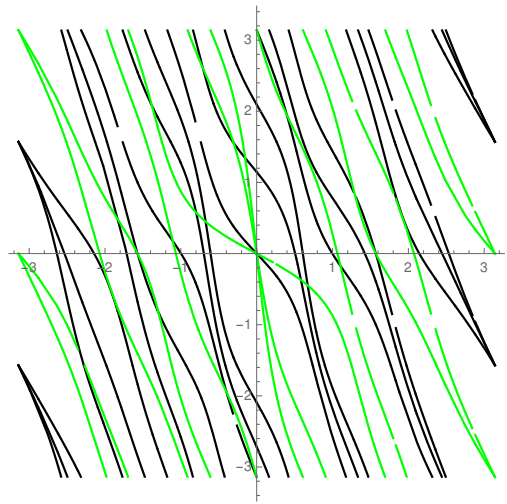}}
  \caption{\textsl{The zero set, a level set, and parameterizing functions for $\phi$ with denominator
   \eqref{combofeature}.}}
  \label{combocurveplots}
\end{figure}

Let us now consider the integrability of $\frac{\partial \phi}{\partial z_3}$ near these components of $\mathcal{Z}_p \cap \mathbb{T}$. First observe that
\begin{multline*}
\psi^0(z_1,z_2)=\frac{1}{192 z_1^6 z_2^2} \left(-1 - 6 z_1^2 - 8 z_1^3 - 9 z_1^4 - 24 z_1^5 - 16 z_1^6 - 2 z_2 - 
   12 z_1^2 z_2 - 16 z_1^3 z_2 - 18 z_1^4 z_2\right.\\ \left.- 48 z_1^5 z_2 - 32 z_1^6 z_2 - 3 z_2^2 - 
   18 z_1^2 z_2^2 + 24 z_1^3 z_2^2 - 27 z_2^4 z_2^2 + 72 z_1^5 z_2^2 - 48 z_1^6 z_2^2\right).
\end{multline*}
A careful analysis reveals that the quadratic form at $(0,0)$ associated with $\rho_{\phi}$ from \eqref{eqn:psi} is
\[Q_{\phi}(\theta_1,\theta_2)=3\theta_1^2+\tfrac{2}{9}\theta_2^2,\]
a pure sum of squares. Then Theorem \ref{thm:GT2} implies that $\frac{\partial \phi}{\partial z_3}\in L^{\p}$ locally near $(1,1,-1)$ if and only if $\p<2$.

The global integrability of $\frac{\partial  \phi}{\partial z_3}$ on $\mathbb{T}^3$ is worse, however. By Lemma \ref{lemma:Qformlemma}, we have $\nabla \rho_{\phi}(\pi, \theta_2)=\vec{0}$, and one checks that
$\frac{\partial^2\rho_{\phi}}{\partial \theta_2^2}(\pi, \theta_2)=\frac{\partial^2\rho_{\phi}}{\partial \theta_1\partial \theta_2}(\pi, \theta_2)=0$
while 
\[\frac{\partial^2 \rho_{\phi}}{\partial \theta_1^2}(\pi, \theta_2)=\frac{3}{4}\left(5+2\cos \theta_2+\cos(2\theta_2)\right)>0.\]
Thus, along the curve singularity, $\rho_{\phi}$ has a Taylor expansion with first non-vanishing term having degree two.
In particular, expanding $|\psi^0(e^{i\theta_1},e^{i\theta_2})|^2$ at, say, $(\pi, 0)$ yields
\[\rho_{\phi}(\pi-\eta_1,\eta_2)=3\eta_1^2+\mathcal{O}(\|\eta\|^4), \]
a quadratic form that is positive semi-definite but not strictly positive definite. Thus \eqref{eqn:psi} and the discussion in Remark \ref{rem:Green} can again be used to deduce that $\frac{\partial \phi}{\partial z_3}\in L^{\p}(\T^3)$ if and only if $\p<\frac{3}{2}$.  However, because the quadratic form is not positive definite, this is another instance where a direct application of Greenblatt's Theorem \ref{thm:Green} would not yield the optimal integrability index.
\end{example}

\begin{example} \label{ex:curveiso2}

%Here we exhibit a rational inner function that has an isolated singularity and worse integrability properties for 
%$\frac{\partial \phi}{\partial z_3}$ than \eqref{ex:3dfave}. We were not able to locate such a function with degree $(m,n,1)$ and
%so, present an example here with degree $(2, 2, 2).$

To construct this example, we use a glueing procedure analogous to that presented in  \cite[Section 6]{BPS19}.
Specifically, let $p$ and $\tilde{p}$ be the denominator and numerator in \eqref{ex:3dfave}, respectively, and set $r(z)=p(z)^2+\tilde{p}(z)^2$. Take
\[\tilde{q}(z)=\sum_{j=1}^3z_j\tfrac{\partial r}{\partial z_j}(z),\] 
and reflect to obtain the polynomial $q$: 
%\begin{multline}
%q(z_1,z_2,z_3)=-6 - 4 z_1 + 2 z_1^2 - 4 z_2 + 4 z_1 z_2 + 2 z_1^2 z_2 + 2 z_2^2\\ + 
% 2 z_1 z_2^2 - 2 z_1^2 z_2^2 + 2 z_3 + 2 z_1 z_3 + 2 z_1^2 z_3 + 2 z_2 z_3 + 
% 4 z_1 z_2 z_3 - 6 z_1^2 z_2 z_3 + 2 z_2^2 z_3\\ - 6 z_1 z_2^2 z_3 - 
% 6 z_1^2 z_2^2 z_3 + 2 z_1^2 z_3^2 + 4 z_1 z_2 z_3^2 - 18 z_1^2 z_2 z_3^2 + 
% 2 z_2^2 z_3^2 - 18 z_1 z_2^2 z_3^2 + 36 z_1^2 z_2^2 z_3^2
%\label{ex:glueptfave}
%\end{multline}
%and
\begin{multline}
q(z_1,z_2,z_3)=36 - 18 z_1 + 2 z_1^2 - 18 z_2 + 4 z_1 z_2 + 2 z_2^2 - 6 z_3 - 6 z_1 z_3 + 
 2 z_1^2 z_3 - 6 z_2 z_3 \\+ 4 z_1 z_2 z_3 + 2 z_1^2 z_2 z_3 + 2 z_2^2 z_3 + 
 2 z_1 z_2^2 z_3 + 2 z_1^2 z_2^2 z_3 - 2 z_3^2 + 2 z_1 z_3^2 \\ + 2 z_1^2 z_3^2 + 
 2 z_2 z_3^2 + 4 z_1 z_2 z_3^2 - 4 z_1^2 z_2 z_3^2 + 2 z_2^2 z_3^2 - 
 4 z_1 z_2^2 z_3^2 - 6 z_1^2 z_2^2 z_3^2.
\label{ex:gluepfave}
\end{multline}
We have $q(1,1,1)=0$ by direct computation. By arguing as in \cite[Section 6]{BPS19}, or by direct substitution, we see that the level set $\mathcal{C}_{-1}$ coincides with the union of the $i$-level set and the $-i$-level set of the RIF in Example \ref{ex:can}.
In particular, $\mathcal{C}_{-1}$ consists of two smooth sheets meeting at $(1,1,1)$ only. We next check that for  $(\zeta_1,\zeta_2)\in \mathbb{T}^2\setminus\{(1,1)\}$, we have $q(\zeta_1,\zeta_2, \Psi_{\pm i}(\zeta_1,\zeta_2))\neq 0$, where $\Psi_{\lambda}$ is as in Example \ref{ex:faverevisited}. Thus $\phi = \frac{\tilde{q}}{q}$ is a degree $(2,2,2)$ rational inner function with an isolated singularity at $(1,1,1)$ by Theorem \ref{thm:CL}.  One can also check that $\phi^*(1,1,1)=-1$. However, because $\deg \phi \ne (m,n,1)$, we cannot use the results in Section \ref{sec:3var2} to study the non-tangential boundary values and unimodular level sets of $\phi$. Instead, we restrict to considering its derivative integrability properties.

\begin{figure}[h!]
 \includegraphics[width=0.4 \textwidth]{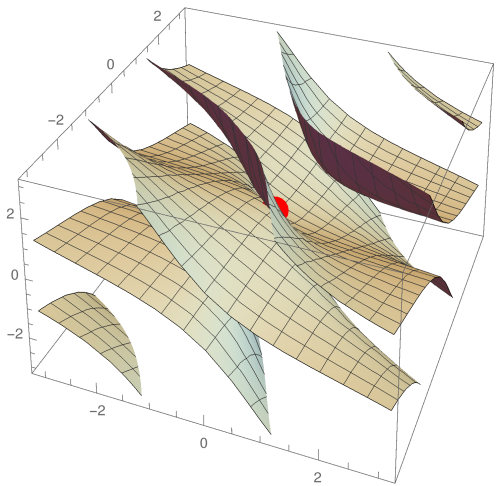}
  \caption{\textsl{Level set $\mathcal{C}_{-1}$, made up of two sheets, for the RIF with denominator \eqref{ex:gluepfave}}.}
  \label{fig:gluesurface}
\end{figure}

First, solving $\tilde{q}(z)=0$ for $z_3$ gives us the two functions
\begin{multline*}
\psi^0_1(z_1,z_2)=\frac{1}{2 (2 + 4 z_1^2 + 8 z_1 z_2 - 30 z_1^2 z_2 + 
     4 z_2^2 - 30 z_1 z_2^2 + 54 z_1^2 z_2^2)}\\ \cdot \Big[6 - 4 z_1 - 4 z_2 - 8 z_1^2 z_2 - 8 z_1 z_2^2 + 
   30 z_1^2 z_2^2 - \left((-6 + 4 z_1 + 4 z_2 + 8 z_1^2 z_2 + 8 z_1 z_2^2 - 
        30 z_1^2 z_2^2)^2\right.  \\  \left. - 
      4 (-6 z_1 + 2 z_1^2 - 6 z_2 + 4 z_1 z_2 + 2 z_2^2 + 4 z_1^2 z_2^2) (2 + 
         4 z_1^2 + 8 z_1 z_2 - 30 z_1^2 z_2 + 4 z_2^2 - 30 z_1 z_2^2 + 
         54 z_1^2 z_2^2)\right)^{\frac{1}{2}} \Big]
\end{multline*}
and
\begin{multline*}
\psi^0_2(z_1,z_2)=\frac{1}{2 (2 + 4 z_1^2 + 8 z_1 z_2 - 30 z_1^2 z_2 + 
     4 z_2^2 - 30 z_1 z_2^2 + 54 z_1^2 z_2^2)}\\ \cdot \Big[6 - 4 z_1 - 4 z_2 - 8 z_1^2 z_2 - 8 z_1 z_2^2 + 
   30 z_1^2 z_2^2 + \left((-6 + 4 z_1 + 4 z_2 + 8 z_1^2 z_2 + 8 z_1 z_2^2 - 
        30 z_1^2 z_2^2)^2\right.  \\  \left. - 
      4 (-6 z_1 + 2 z_1^2 - 6 z_2 + 4 z_1 z_2 + 2 z_2^2 + 4 z_1^2 z_2^2) (2 + 
         4 z_1^2 + 8 z_1 z_2 - 30 z_1^2 z_2 + 4 z_2^2 - 30 z_1 z_2^2 + 
         54 z_1^2 z_2^2)\right)^{\frac{1}{2}} \Big].
\end{multline*}
Note that we need two functions $\psi^0_{1}, \psi^0_2$ because $\deg \phi = (2,2,2)$. Direct substitution reveals that $\psi^0_2(1,1)=1$ and $\psi^0_1(1,1)=0$. Hence only the branch parametrized by $\psi^0_2$ hits the singular point $(1,1,1)$. Thus to study the integrability of $\frac{\partial \phi}{\partial z_3}$, it suffices to consider 
\[\rho_{\phi}(\theta_1,\theta_2)=1-|\psi^0_2(e^{i\theta_1},e^{i\theta_2})|^2.\]
A careful analysis shows that
\[\rho_{\phi}(\theta_1,\theta_2)=\theta_1^4+2\theta_1^3\theta_2+3\theta_1^2\theta_2^2+2\theta_1\theta_2^3+\theta_2^4+\mathcal{O}(\|\theta\|^5),\]
for $(\theta_1, \theta_2)$ near $(0,0)$. 
This means that $\frac{\partial \phi}{\partial z_3}\in L^{\p}(\T^3)$ for the same range of $\p$ for which
\begin{equation}
\iint_{U} (\theta_1^4+2\theta_1^3\theta_2+3\theta_1^2\theta_2^2+2\theta_1\theta_2^3+\theta_2^4)^{1-\p}d\theta_1d\theta_2<\infty
\label{deg4integral}
\end{equation}
for some neighborhood $U\ni (0,0)$. Setting $\mathcal{Q}(\theta_1,\theta_2)=\theta_1^4+2\theta_1^3\theta_2+3\theta_1^2\theta_2^2+2\theta_1\theta_2^3+\theta_2^4$, we note that
\[\mathcal{Q}(x+y,x-y)=(3x^2+y^2)^2.\]
Thus, making this change of variables, followed by the scaling $x\mapsto \sqrt{3}x$ and $y\mapsto y$, we deduce that \eqref{deg4integral} is finite if and only if
\[\iint_{\tilde{U}} (x^2+y^2)^{2-2\p}dxdy\]
is finite for some neighborhood $\tilde{U}$ of $(0,0)$. Introducing polar coordinates transforms the latter integral condition to the requirement that
\[\int_0^{\epsilon}r^{5-4\p}dr\]
be finite, which happens precisely when $\p<\frac{3}{2}$.

Comparing to Example \ref{ex:curve}, we see  that this isolated-singularity RIF  has the same integrability range for its $z_3$-derivative as a RIF with a curve of singularities in $\T^3$.
\end{example}

\begin{question}
Is there a degree $(m,n,1)$ RIF with an isolated singularity and the same integrability behavior as in Example \ref{ex:curveiso2}?
\end{question}

%\section{Open Questions}

\subsection*{Acknowledgments}
We thank our home institutions, Bucknell University, the University of Florida, and Stockholm University, for facilitating visits during which large parts of this work took shape. 
 
\end{document}